\documentclass[12pt, a4paper, reqno]{amsart}
\usepackage{amsthm, mathtools, amssymb, amsfonts, bm, stmaryrd, latexsym, mathrsfs, todonotes, eufrak, color, tikz-cd}
\usepackage{graphicx}
\usepackage[english]{babel}
\usepackage[all]{xy}
\usepackage{nicefrac}
\usepackage{mparhack}
\usepackage[normalem]{ulem}
\usepackage{color}
\usepackage{blindtext}
\usepackage[inline]{enumitem}
\usepackage{fancyhdr}
\usepackage{enumitem}
\usepackage{algorithmic}
\usepackage{algorithm}
\usepackage{hyperref}
\usepackage{tikz-cd}
\usepackage{graphicx}
\usepackage[symbol*]{footmisc}

\newcommand{\rotcong}{\rotatebox{90}{$\cong$}}

\definecolor{mycol}{rgb}{5,0,0}

\theoremstyle{plain}
\newtheorem{theo}{Theorem}[section]
\newtheorem{lemma}[theo]{Lemma}
\newtheorem{prop}[theo]{Proposition}
\newtheorem{corollary}[theo]{Corollary}
\newtheorem{remark}[theo]{Remark}
\newtheorem{example}[]{Example}

\theoremstyle{definition}

\newcommand{\F}{\mathbb{F}}

\newcommand{\Z}{\mathbb{Z}}
\newcommand{\Q}{\mathbb{Q}}
\newcommand{\R}{\mathbb{R}}

\newcommand{\N}{\mathbb{N}}

\newcommand{\cm}{\mathcal{M}}
\newcommand{\M}{\mathcal{M}}
\newcommand{\ok}{\mathcal{O}_K}

\newcommand{\rank}{{\rm rank}}

\renewcommand{\epsilon}{\varepsilon}

\newcommand{\ord}{\mathop{\mathrm{ord}}\nolimits}

\newcounter{enumi_saved}

\makeatletter
\def\imod#1{\allowbreak\mkern10mu({\operator@font mod}\,\,#1)}
\makeatother

\allowdisplaybreaks

  \title{Finitely generated abelian groups of units}
\begin{document}
\subjclass[2010]{16U60, 20K15, 13A99}
\keywords{
Commutative algebra;
Group theory;
Fuchs' Problem;
Units groups;
Torsion-free  rings;
Cyclotomic extensions;
Cyclotomic polynomials.
}
\maketitle
\begin{center}\author{\textsc{Ilaria Del Corso\footnote{ Dipartimento di Matematica Universit\`a di Pisa,\\ e-mail: ilaria.delcorso@unipi.it} }}
\end{center}

\begin{abstract}
In
%\cite[Problem 72]{Fuchs60} 
1960 Fuchs posed the problem of characterizing  the  groups which are the groups of units of  commutative rings. 
In the following years, some partial answers have been given to this question in particular cases.

In this paper we address Fuchs' question for {\it finitely generated abelian} groups and we consider the problem of characterizing those groups which arise in some  fixed classes of rings $\mathcal C$, namely the integral domains, the torsion free rings and the reduced rings.

%To determine the realizable groups we have to establish what finite abelian groups $T$  (up to isomorphism) occur as torsion subgroup of $A^*$ when $A$ varies in $\mathcal C$, and on the other hand, we have to determine what are the possible values of the rank of $A^*$ when $(A^*)_{tors}\cong T$. 

Most of the paper is devoted to the study of the class of torsion-free rings, which needs a substantially deeper study.
 \end{abstract}
 
\section{Introduction}
\subsection{General introduction to the problem}
The study of the group of units of a ring is an old problem. The first general result is the classical Dirichlet's Unit Theorem (1846), which describes the group of units of the ring of integers $\ok$ of a number field $K$: the group of units $\ok^*$ is a finitely generated abelian group of the form $C_{2n}\times\Z^g$ where $n\ge1$ and $g$ is 
%explicit in terms of 
determined by
the structure of the field $K$.

In 1940 G. Higman discovered a perfect analogue of Dirichlet's Unit Theorem for a group ring $\Z T$ where $T$ is a finite abelian group: $(\Z T)^*\cong\pm T\times\Z^g$ for a suitable explicit constant $g$.

In 1960 Fuchs in  \cite[Problem 72]{Fuchs60} 
%raised explicitly the question  posing 
posed
the following problem.

\begin{itemize}
\item[]{\sl Characterize the groups which are the groups of all units in a commutative and associative ring with identity.}
\end{itemize}
In the subsequent years, this question has been considered by many authors.
A first result is due to  Gilmer \cite{Gilmer63}, who considered the case of {\it finite commutative rings}, classifying  the possible cyclic groups that arise in this case. 
An important  contribution to the problem can be derived from the results by  Hallett and Hirsch \cite{HallettHirsch65}, and subsequently by Hirsch and Zassenhaus
\cite{HirschZassenhaus66}, combined with  \cite{Corner63}. From their study it is possible to deduce that if  a finite group is  the group of units  of a reduced and  torsion free ring, then it must satisfy some necessary conditions, namely, it must be a subgroup of a direct product of groups of a given  family.

Later on, Pearson and Schneider \cite{PearsonSchneider70} combined the result of Gilmer and the result 
of Hallett and Hirsch to describe explicitly all possible {\it finite cyclic groups} that can occur as $A^*$ for 
%a commutative 
some ring $A$. 

 Recently, Chebolu and Lockridge \cite{CheboluLockridge15} were able to 
classify the {\it indecomposable abelian groups}  which occur as groups of units of a ring. 

In the papers \cite{DCDampa} \cite{dcdBLMS} R. Dvornicich and the author studied Fuchs' question for {\it finite abelian groups} and for a general ring of any characteristic, obtaining necessary conditions  for a group to be realizable, and  producing  infinite families of both realizable and non-realizable groups. Moreover, they got a complete classification of the group of units realizable in some  particular classes of rings (integral domains, torsion-free rings and reduced rings).

The study  of groups of  units has been investigated also for non abelian groups. Much has been said about the units of group rings. Recently, the  finite dihedral groups and the simple groups   that are realizable
as the group of units of a ring  have been classified (see  \cite{CheboluLockridge17} and \cite{DavisOcchipinti2014}).

\subsection{The questions studied in the paper}
In this paper we consider Fuchs' question for {\it finitely generated abelian} groups and we consider the problem of characterizing  those groups which arise in some fixed classes of rings $\mathcal C$, namely the integral domains, the torsion free rings and the reduced rings.

This question is twofold: on the one hand, we have to establish 
%what 
which
finite abelian groups $T$  (up to isomorphism) occur as the torsion subgroup of $A^*$ when $A$ varies in $\mathcal C$. On the other hand, we have to determine 
%what are 
the possible values of the rank of $A^*$ when $(A^*)_{tors}\cong T$. 
Therefore, the situation becomes substantially different from the case when the group of units is finite and abelian, which has  been studied already in \cite{DCDampa} and \cite{dcdBLMS}.
  
\subsection{Integral domains: result and idea of proof} In Section \ref{integral-domains} we focus on the study of groups of units of integral domains. Our main tools are Dirichlet's Unit Theorem and the properties of cyclotomic extensions. The principal result is the following theorem in which we collect the results of Theorems \ref{domini-char0} and \ref{domini-charp}.

{\bf Theorem A: }
{\sl
 The  finitely generated abelian groups that occur as groups of units of integral domains are:
 
 i) the groups of the form $C_{2n}\times\Z^g$, with $n\in\N$, $g\ge\frac{\phi(2n)}2-1$, for domains of characteristic zero;
 
 ii) the groups  of the form $\F_{p^n}^*\times\Z^g$ with $n\ge1$ and $g\ge0$, for domains of finite characteristic.
 }
\smallskip 

As a particular case we get the characterization of the finite abelian groups which are realizable as group of units of an integral domain  (see Corollary \ref{cor-domini}).

Finally, in Proposition \ref{domini-int} we describe the finitely generated abelian groups that occur as group of units of an integral domain $A$ which is integral over $\Z$.

\subsection{Torsion-free rings: result and idea of proof}
The most relevant part of the paper  is the classification of the finitely generated abelian groups of units realizable with torsion-free rings (Sections \ref{sec:torsion-free-prel} and \ref{sec:torfree}). We remark that the study of the group of units of torsion free rings has become classical in the literature (see the aforementioned papers by Hallett, Hirsch and Zassenhaus) and that   the finitely generated abelian group rings belong to this class.

In Theorem \ref{torfree} we prove the following

{\bf Theorem B: }
{\sl
 Let $T$ be a finite abelian group of even order. Then there exists an explicit constant $g(T)$ depending on  $T$ (see \eqref{eqgT} for the explicit value of $g(T)$) such that the following holds:
 the group $T\times\Z^r$ is the group of units of a torsion free ring 
%there exists a torsion free ring $A$ with
%$$A^*\cong T\times\Z^r$$
 if and only if  $r\ge g(T).$}
 \smallskip

 The proof is rather long and requires many steps.
 The first step is the reduction to the study of the subring of $A$ generated over $\Z$ by the torsion units. This ring has the same torsion units as $A$ and is finitely generated and integral over $\Z$. Restricting to study  these rings, in Proposition \ref{QB} we show that the $\Q$-algebra $A\otimes_\Z\Q$ is  semisimple and is a finite product of cyclotomic fields (for short, a cyclotomic $\Q$-algebra).
 The next 
 %important 
 step is the study of the  units of the  subrings of $A$ of type $\Z[\alpha]$, with $\alpha$ a torsion unit of $A$, in some particular cases  (see Propositions \ref{prop:zetaalpha} and
 \ref{units_1n}).
 Once  these preliminary results are established, we pass to the proof of the theorem, which requires two parts.
 
 On the one hand, we have to show that if $A$ is a torsion-free ring with $(A^*)_{tors}\cong T$, then $\rank(A^*)\ge g(T)$. This is done through  the analysis of the possible maximal order of $T$-admissible cyclotomic $\Q$-algebras (namely, cyclotomic $\Q$-algebras which could admit a subring with $(A^*)_{tors}\cong T$). This gives  a first lower bound on the rank of the group of units (Proposition \ref{gT}).   This ``natural'' bound works only if the 2-Sylow subgroup of $T$ has ``enough'' cyclic factors of minimal order in its decomposition. If not, the actual bound is bigger than the natural one: 
% in fact, in this case it is necessary to consider a bigger maximal order obtained by adding to the natural candidate an extra direct factor which works as a ``control" factor. T
 this is described in Proposition \ref{nom0T}.
 % and \ref{sigma=1}.    
  
  On the other hand, for each $T$ we have to construct a torsion-free ring $A$ with $A^*\cong T\times \Z^{g(T)}$: the construction of orders with a bigger rank can then be obtained via localization.
  In the previous part for a given $T$ we have identified  a maximal order $\cm_T$ of a cyclotomic $\Q$-algebra with $\rank(\cm_T^*)=g(T)$.  We construct $A$ as an order of $\cm_T$, hence $\rank(A)=\rank(\cm_T^*)=g(T)$ (see Lemma \ref{finito}). The group $(\cm_T^*)_{tors}$ contains a subgroup isomorphic to $T$ and it differs  from $T$ only in the 2-Sylow subgroup: our  task is  to construct an order with a 2-Sylow as small as possible. 
  
We note that also in this case the results of \cite{dcdBLMS} on finite abelian groups of units are recovered as a corollary of this more general result.

 \subsection{Reduced rings: result and idea of proof}
In  Section \ref{reduced} we deal with the units of reduced rings. 
%We not that the study of  this class  is also  a step towards the understanding of the units of 
%non-reduced rings. In fact, for a ring $R$  with nilradical $\mathcal N$ it is known that $R^*$ is an extension of $(R/\mathcal N)^*$ by $1+\mathcal N$ (see Proposition \ref{successioneesatta}). 
%

For a non reduced ring $R$  with nilradical $\mathcal N$,  it is known that the  $R^*$ is an extension of $(R/\mathcal N)^*$ by $1+\mathcal N$ (see Proposition \ref{successioneesatta}). So 
 the study of  units of reduced rings  is also a step towards the understanding of the units of 
general rings. 

%So in a certain (weak) sense it is permitted to
%restrict attention to reduced rings. 

In Theorem \ref{reduced-all} we prove the following.
 
 {\bf Theorem C: }
{\sl
The  finitely  generated abelian groups that occur as group of units of a reduced ring are those of the form
$$\prod_{i=1}^k\F_{p_i^{n_i}}^*\times T\times \Z^g$$
where $k, n_1,\dots, n_k$ are positive integers, $\{p_1,\dots,p_k\}$ are  not necessarily distinct primes, $T$ is any finite abelian group of even order and $g\ge g(T)$.}
\smallskip

The proof is achieved by using a result by Pearson and Schneider \cite[Prop.~1]{PearsonSchneider70} which allows one to split a generic reduced ring $A$ as a direct sum  $A_1\oplus A_2$ where $A_1$ is finite and $A_2$ is torsion-free. Putting together our previous results on torsion-free rings  with some properties of the finite rings we get the classification of the groups of units in this case.

{\it Acknowledgement:} I  wish to thank  Cornelius Greither for his careful reading of the paper and for suggesting to me stylistic improvements and a refinement of the proof of Proposition \ref{QB}. I wish also to warmly thank the anonymous referees for their careful  reading of the paper. Their suggestions have been fundamental for improving the readability of the paper.

%%%%%%%%%%%%%%%%%%%%%%%%%%%%%%%%%%%%%%%%%%
\section{Notation and preliminary results}
\label{notation}
Let $A$ be a ring with 1: throughout  the paper we will assume that its group of units  $A^*$ is finitely generated and abelian. Let $(A^*)_{tors}$ denote its torsion subgroup and let $g_A$ be its rank so that 
$$A^*\cong (A^*)_{tors}\times \Z^{g_A}.$$

  Let $A_0$ be the fundamental subring of $A$, namely  $A_0=\Z$ or $\Z/n\Z$ depending on whether the characteristic of $A$ is 0 or $n$. It is immediate to check that the ring $A_0[A^*]$ has the same group of units as $A$.  Since we are interested in the classification of the possible groups of units, we can  assume without loss of generality  that $A$ is a ring of type $A_0[A^*]$. In particular, we will always assume that $A$ is commutative and that it is finitely generated over $A_0$.     
  
Let $B$ the subring of $A$ generated over $A_0$ by the torsion units of $A$, namely $B\cong A_0[(A^*)_{tors}]$.   
It is important to note that all the elements of $(A^*)_{tors}$ are integral over $A_0$, since they have  finite order. This ensures that $B$ is {\em commutative, finitely generated and integral over $A_0$}. 

\begin{lemma}
\label{minimum}
{\sl $B^*\cong (A^*)_{tors}\times \Z^{g_B}$ and $g_B\le g_A$. Moreover, if the characteristic of $A$ is positive, then $B^*=(A^*)_{tors}$.
}
\end{lemma}
\begin{proof}
$B$ is a subring of $A$, hence $B^*<A^*$: in particular $B^*$ is finitely generated and $g_B\le g_A$. On the other hand, $(A^*)_{tors}<(B^*)_{tors}<(A^*)_{tors}$ and equality holds.

Moreover, when the characteristic of $A_0$ is positive, then $B$, being integral and finitely generated over $A_0$, is itself finite, so $B^*=(A^*)_{tors}$.
\end{proof}
\begin{remark}
\label{rem-differenze}
{\rm
The previous lemma shows that all possible torsion parts occur already when restricting to consider rings which are generated over $A_0$ by a finite number of integral elements verifying an equation of type $x^n-1$ for some $n$.

The lemma also shows that there is a completely different behavior between the  characteristic zero and positive characteristic rings. In fact,  a finite abelian group $T$ can be isomorphic to the torsion subgroup of the group of units of a ring $A$ of positive characteristic only if it is also the group of units of a finite ring and all the results of \cite{DCDampa} apply in this case. In particular,  not all finite abelian groups can occur.

Instead, when $A_0=\Z$  it will turn out that the torsion subgroup of $A^*$ can be any finite abelian groups of even order, whereas this is not true if we also require that $A^*$ is finite (see Theorem \ref{torfree} and \cite{dcdBLMS}).
Nevertheless,  to determine the minimum rank $g(T)$  such that $T\times \Z^{g(T)}$ is the group of units of some ring $A$, it is sufficient 
%to restrict 
to consider the  finitely generated integral extensions of $\Z$. 
}
\end{remark}

In the following subsections we collect some classical results we will need in the paper.

\subsection{Units of Laurent polynomials}  Let $R$ be a reduced ring, namely a ring without non-zero nilpotents. Then the polynomial ring $R[x]$ is reduced and has the same units as $R$ and the ring of Laurent polynomials $R[x, x^{-1}]$
%=S^{-1}R[x]$  with $S=\{x^n\}_{n\in\N}$ 
has group of units $\langle R^*, x\rangle$.
Inductively we get that the  group of units of the ring of Laurent polynomials in $k$ indeterminates
$R[x_1\dots, x_k, x_1^{-1},\dots, x_k^{-1}]$ is  isomorphic to $R^*\times\Z^k$.

  \subsection{The Chinese Remainder Theorem}
  \label{CRT}
  Let $R$ be a commutative ring with 1 and let $I,J\subseteq R$ be ideals.
  Then the map
  $$\psi\colon R\to R/I\times R/J$$
  defined by $r\mapsto(r+I, r+J)$ is a ring homomorphism with kernel $I\cap J$ and image $\{(r+I, s+J)\mid r-s\in I+J\}$. The well known Chinese Remainder Theorem (in the following CRT) ensures that $\psi$ is surjective if and only if $I+J=R$.
  
  More generally, if $I_1,\dots, I_n$ are ideals of $R$ we can define the homomorphism
  $$\psi\colon R\to R/{I_1}\times\dots\times R/{I_n}$$
  by $\psi(r)=(r+I_1,\dots, r+I_n)$.
We will refer to the map $\psi$ or to the map  induced by $\psi$ on $R/\cap_{i=1}^nI_n$ as to the CRT map.
  
  \smallskip
  
  In the following we will consider the CRT map when $R=\Z[x]$, $I=(f(x))$ and $J=(g(x))$. If $f(x)$ and $g(x)$ are coprime polynomials, then $I\cap J=IJ=(f(x)g(x))$, so the CRT map   $$\psi\colon\Z[x]/(f(x)g(x))\to\Z[x]/(f(x))\times\Z/(g(x))$$ 
  is an injection and it is an isomorphism if and only if $(f(x),g(x))=\Z[x]$.

  \subsection{Dirichlet's Unit Theorem}
Let $K$ be a number field,  and let $\ok$ be its ring of integers;
the classical Dirichlet's Theorem describes  the groups of units of all  {\em orders} of $K$ (we recall that 
an { order} of $K$ is a subring of $\ok$ which spans $K$ over $\Q$). 
\begin{prop}[Dirichlet's Unit Theorem]
\label{dirichlet}
{\sl
Let $K$ be a number field such that $[K:\Q]=n$ and 
assume that among  the $n$ embeddings of $K$ in $\bar\Q$, $r$ are real (namely map $K$ into $\R$) and $2s$ are non-real ($n=r+2s$). Let $R$ be an order of $K$. Then
$$R^*\cong T\times\Z^{r+s-1}$$
where $T$ is the group of the roots of unity contained in $R$.
}
\end{prop}
For a proof  see  \cite[Ch.1,\textsection12]{Neukirch99}. 

%Dirichlet's units Theorem admits the following genaralization (see for example \cite{Janusz96} p.207).
%\begin{prop}[S-units Theorem]
%\label{S-units}
%{\sl Let $K$ be a number field such that $[K:\Q]=n=r+2s$ and let $S$ be a finite set of prime ideals of $\ok$ and let ${\mathcal O}_{K,S}=\{\alpha\in K\mid v_P(\alpha)\ge0 \ \forall P\not\in S\}$. Then 
%$${\mathcal O}_{K,S}^*\cong T\times \Z^{r+s-1+|S|}$$
%where $T$ is the group of the roots of unity contained in $K$.
%
%}
%\end{prop}

\subsection{Cyclotomic polynomials}
\label{subs:cyclo}

 For $n\ge1$ let   $\zeta_n=e^{2\pi i/n}$, then $\zeta_n$ is  a primitive $n$-th root  of unity. Denote  by $\Phi_{n}(x)$ its minimal polynomial over $\Q$: as it is well known, $\Phi_n(x)\in\Z[x]$ and
 $$\Phi_n(x)=\prod_{\substack{j=1,\dots,n\\ (j,n)=1}}(x-\zeta_n^j).$$
 Moreover, $\Q(\zeta_n)$ is a Galois extension of $\Q$ of degree $\phi(n)$, where $\phi$ is the Euler totient function, and its ring of integers is $\Z[\zeta_n]$. 
 
The roots of unity contained in  $\Z[\zeta_n]$ are the $n$-th roots of unity if $n$ is even and the $2n$-roots of unity if $n$ is odd and by Dirichlet's  Unit Theorem 
 $\Z[\zeta_n]^*\cong \langle-\zeta_n\rangle\times \Z^{\frac{\phi(n)}2-1}$ for each $n\ge3.$ 
In the following we will use the notation $(\frac{\phi(n)}2-1)^*$ for the rank of $\Z[\zeta_n]^*$, namely, $(\frac{\phi(n)}2-1)^*=\frac{\phi(n)}2-1$ for $n\ge3$ and $(\frac{\phi(n)}2-1)^*=0$ for $n=1,2$. We will omit the $*$ when $n>2$.

%The following lemma describes how the ideals generated by rational primes factor into the cyclotomic rings.
% \begin{lemma}
%\label{fatt-cyclo}
%{\sl
%Let  $p$ be a prime number and let $e\ge1$Then $(1-\zeta_{p^e})$ is a prime ideal of $\Z[\zeta_n]$ and $p\Z[\zeta_{p^e}]=(1-\zeta_{p^e})^{\phi(p^e)}$.
%}
%\end{lemma}
%\begin{proof}
%See \cite[Thm 1]{Lang94ANT}. 
%\end{proof}

In 
%the 
this
paper we will need the following classical property of cyclotomic fields and cyclotomic polynomials. 
Most of the results could be generalized, but we give only  those necessary for our purposes.

 \begin{lemma}
 \label{lemmawashington}
 .
 
   {\sl
    \begin{enumerate}
    \item Suppose that $n$ has at least two distinct prime factors. Then $1-\zeta_n$ is a unit of $\Z[\zeta_n]$ and 
   $$\Phi_n(1)=\prod_{\substack{j=1,\dots, n\\ (j,n)=1}}(1-\zeta_n^j)=1.$$
   \item For $p$ prime and $e>0$, then $1-\zeta_{p^e}$ is a generator of the prime ideal of $\Z[\zeta_{p^e}]$ lying over $(p)$,   $$p\Z[\zeta_{p^e}]=(1-\zeta_{p^e})^{\phi(p^e)}$$ and $$\Phi_{p^e}(1)=\prod_{\substack{j=1,\dots,p^e\\ (j,p^e)=1}}(1-\zeta_{p^e}^j)=p.$$
   \end{enumerate}
   } 
   \end{lemma}
   \begin{proof}
   For part (1) see \cite[Lemma 2.8]{washington}. 
   For part (2) see \cite[IV, 1, Thm 1]{Lang94ANT}.
\end{proof}
\begin{lemma}
\label{Psin1}
{\sl
Let $l>1$ and let $\Psi_{n,l}(x)$ denote the minimal polynomial of $\zeta_n$ over  $K=\Q(\zeta_l)$. 
\begin{enumerate}
    \item Suppose that $n$ has at least two distinct prime factors. Then  the algebraic integer $\Psi_{n,l}(1)$ is a unit.
   \item If $n=p^a$, where $p$ is  a prime and $a>0$, and $l=l_1p^b$, with $(l_1, p)=1$ and $0\le b\le a$, then $\Psi_{p^a, l}=\Psi_{p^a, p^b}$ and $\Psi_{p^a, l}(1)$
is a generator of the prime ideal of $\Z[\zeta_{p^b}]$ lying over $(p)$. 
 \end{enumerate}
}
\end{lemma}
\begin{proof}
$\Psi_{n,l}(x)$ divides $\Phi_n(x)$, hence $\Psi_{n,l}(1)$ is a unit since it divides the unit $\Phi_n(1)$ (this actually holds for any number field $K$).

For part $(2)$ note that $\Psi_{p^a, l}(x)=\Phi_{p^a}$ if $b=0$ and $\Psi_{p^a, l}(x)=\Psi_{p^a, p^b}(x)=x^{p^{a-b}}-\zeta_{p^b}$ if $b>0$ (a divisibility relation is obvious and equality follows from a degree argument). It follows that $\Psi_{p^a, l}(1)$ is equal to $p$ or $1-\zeta_{p^b}$ according to $b=0$ or $b>0$, namely it is a generator for the prime ideal of $\Z[\zeta_{p^b}]$ lying over $(p)$.
\end{proof}

\begin{lemma}
\label{phinzetam}
{\sl
Let $n>m\ge1$. The algebraic integer $\Phi_n(\zeta_m)$ is a unit in $\Z[\zeta_m]$  if $n/m$ is not a prime power. 

In the case when  $n/m=p^a$ for a prime $p$ and an integer $a>0$, then $\Phi_n(\zeta_m)$ is associated to $p$. 
}
\end{lemma}
\begin{proof}
%The first part of  is an immediate  consequence of  \cite[Thms 1 and 4 ]{apostol70}. 
The first part of the proof is \cite[Corollary 8]{moree} (see also  \cite{apostol70}). 

For the second part, we note that 
$$\Phi_n(\zeta_m)= \prod_{\substack{j=1,\dots,n\\ (j,n)=1}}(\zeta_m-\zeta_n^j)=\prod_{\substack{j=1,\dots,n\\ (j,n)=1}}\zeta_m(1-\zeta_n^{j-p^a}).$$
From Lemma \ref{lemmawashington} we have that  $(1-\zeta_n^{j-p^a})$ is invertible if  $\frac{n}{(n, j-p^a)}$ is not a prime power. On the other hand,  $\frac{n}{(n, j-p^a)}$ is a prime power only if it is a power of $p$ and  $j\equiv p^a\pmod{m_1}$, where  $m=p^b m_1$ and $(m_1,p)=1$. Taking into account that $(j,n)=1$, an easy computation shows that there are $\phi(p^{a+b})$ values of $j$ with this property. For these values $1-\zeta_n^{j-p^a}$ is a generator of the ideal $(1-\zeta_{p^{a+b}})$, namely 
$$(\Phi_n(\zeta_m))=(1-\zeta_{p^{a+b}})^{\phi(p^{a+b})}=p\Z[\zeta_{p^{a+b}}].$$
\end{proof}

In Sections \ref{sec:torsion-free-prel} and \ref{sec:torfree} we will need to study
the ring 
$$\Z[x]/(\Phi_{m_1}(x)\dots\Phi_{m_r}(x))$$
 when $m_1, \dots, m_r$ are distinct positive integers. 
Denote by $\psi$ the CRT map, and, by abuse of notation, also its composition with the isomorphism given by 
 the identifications $\Z[x]/(\Phi_{m_i}(x))\cong\Z[\zeta_{m_i}]$, namely 
\begin{equation}\label{CRTmap}
\psi\colon \Z[x]/(\Phi_{m_1}(x)\cdots\Phi_{m_r}(x))\to\prod_{i=1}^r \Z[x]/(\Phi_{m_i}(x))\cong\prod_{i=1}^r \Z[\zeta_{m_i}].
\end{equation}
Then $\psi$ is always an injection and we ask when it is also surjective.

%From the following lemma we get the answer for $r=2$.  

The following lemma gives the answer for $r=2$, in Proposition \ref{prop:cinese} we will give the general answer.
\begin{lemma}
\label{lemma:cycloPsi}
{\sl
Let $n>m{\ge}1$. The following are equivalent:
\begin{enumerate}
\item[i)] $\psi\colon\Z[x]/{(\Phi_m(x)\Phi_n(x))}\to\Z[\zeta_m]\times\Z[\zeta_n ]$ is an isomorphism.
\item[ii)] $(\Phi_m(x), \Phi_n(x))=\Z[x]$;
\item[iii)] $\Phi_n(\zeta_m)$ is invertible;
\item[iv)] $n/m$ is not a prime power.
\end{enumerate} 
}
\end{lemma}

\begin{proof}
(i) is equivalent to (ii) by the CRT.

The equivalence between (ii) and (iii) follows from the 
 following chain of isomorphisms 
$$\frac{\Z[x]}{(\Phi_m(x),\Phi_n(x))}\cong \frac{\Z[x]/(\Phi_m(x))}{(\Phi_m(x),\Phi_n(x))/(\Phi_m(x))}\cong\frac{\Z[\zeta_m]}{(\Phi_n(\zeta_m))}.$$
 Finally,
Lemma \ref{phinzetam} gives the equivalence between (iii) and (iv).
\end{proof}

\begin{prop}
\label{prop:cinese}
{\sl
Let $r\ge 2$ and $m_1<\cdots< m_r$ be distinct positive integers. The following are equivalent
\begin{itemize}
\item[a)] the CRT map \\
$\psi\colon  \Z[x]/(\Phi_{m_1}(x)\cdots\Phi_{m_r}(x))\to\prod_{i=1}^r \Z[x]/(\Phi_{m_i}(x))\cong\prod_{i=1}^r \Z[\zeta_{m_i}]$
is an isomorphism;
\item[b)] for all $1\le i<j\le r$ the ratio $m_j/m_i$ is not a prime power.
\end{itemize}
}
\end{prop}
\begin{proof}
For each $t$ with $2\le t\le r$, consider the  CRT maps
$$
\psi_t\colon \Z[x]/(\Phi_{m_1}(x)\cdots\Phi_{m_t}(x))\to\prod_{i=1}^t \Z[x]/(\Phi_{m_i}(x))
$$
and 
$$
\rho_t\colon\Z[x]/(\Phi_{m_1}(x)\cdots\Phi_{m_t}(x))\to\Z[x]/(\Phi_{m_1}(x)\cdots\Phi_{m_{t-1}}(x))\times \Z[x]/(\Phi_{m_t}(x)).
$$
With this notation,  we have the following commutative diagram
\begin{equation}
\label{commdiag}
\begin{tikzcd}
\Z[x]/( \prod\limits_{i=1}^t\Phi_{m_i}(x))\arrow[d, "\rho_t"]\arrow[r, "\psi_t"]& \prod\limits_{i=1}^t{\Z[x]}/{(\Phi_{m_i}(x))} \\
{\Z[x]}/{( \prod\limits_{i=1}^{t-1}\Phi_{m_i}(x))}\times{\Z[x]}/{(\Phi_{m_t}(x))}\arrow[ur , "\psi_{t-1}\times id"] 
 & 
\end{tikzcd}
\end{equation}
namely, 
\begin{equation}
\label{eq:psit}
\psi_t=(\psi_{t-1}\times id)\circ\rho_t.
\end{equation}

%We consider the statements 
%\begin{itemize}
%\item[a($t$):] $\psi_t$ is an isomorphism;
%\item[b($t$):] for all $1\le i<j\le t$ the ratio $m_j/m_i$ is not a prime power.
%\end{itemize}
% We will prove, by induction on $t$, that a($t$)$\iff$ b($t$). Since a($r$) and b($r$) are  the statements (a) and (b), respectively, this will prove the proposition.
We will prove that (a) is equivalent to (b) by induction on $r$.
 
 For $r=2$ the equivalence is given in Lemma  \ref{lemma:cycloPsi}. We now assume that $r>2$ and that  the equivalence holds in the case of $r-1$ integers $m_1<\cdots<m_{r-1}$ and we prove it for $m_1<\cdots<m_{r-1}<m_r.$
 
 Assume (a), so $\psi=\psi_r$ is an isomorphism. From equation \eqref{eq:psit}
 we get that $\psi_{r-1}\times id$ is surjective; this ensures that also $\psi_{r-1}$ is surjective and hence it is an isomorphism since it is always injective.  Therefore, by inductive hypothesis we get that $m_j/m_i$ is not a prime power for $1\le i<j\le r-1$ and we are left to prove that $m_r/m_i$ is not a prime power for $1\le i<r$.
 We note that  since both $\psi_r$ and $\psi_{r-1}$ are isomorphisms, equation \eqref{eq:psit} ensures that also the CRT map $\rho_r$ is an isomorphism
 so $(\Phi_{m_1}(x)\cdots\Phi_{m_{r-1}}(x), \Phi_{m_r}(x))=\Z[x]$, which in turns implies  $(\Phi_{m_i}(x), \Phi_{m_r}(x))=\Z[x]$ for each $1\le i<r$. By Lemma \ref{lemma:cycloPsi}  the last condition ensures that, for each $i$, the ratio $m_r/m_i$ is not a prime power, proving  (b).

% Let $t>2$ and assume that  a($t$) holds, namely $\psi_t$ is an isomorphism. From equation \eqref{eq:psit}
% we get that $\psi_{t-1}\times id$ is surjective, so $\psi_{t-1}$ is surjective and hence it is an isomorphism since it is always injective. This means that a($t-1$) holds, and hence by inductive hypothesis also b($t-1$) holds proving that $m_j/m_i$ is not a prime power for $1\le i<j\le t-1$. We are left to prove that $m_t/m_i$ is not a prime power for $1\le i<t$.
% We note that  since both $\psi_t$ and $\psi_{t-1}$ are isomorphisms, equation \eqref{eq:psit} ensures that also the CRT map $\rho_t$ is an isomorphism
% and  this is equivalent to $(\Phi_{m_1}(x)\cdots\Phi_{m_{t-1}}(x), \Phi_{m_t}(x))=\Z[x]$, which in turns implies  $(\Phi_{m_i}(x), \Phi_{m_t}(x))=\Z[x]$ for each $1\le i<t$. By Lemma \ref{lemma:cycloPsi}  the last condition ensures that, for each $i$, the ratio $m_t/m_i$ is not a prime power, proving  b($t$).
% 
 Conversely, assume that (b) holds, then by applying the inductive hypothesis to $m_1<\cdots<m_{r-1}$ we get that $\psi_{r-1}$ is an isomorphism. On the other  hand, since   
$m_r/m_i$ is not a prime power, by Lemma  \ref{lemma:cycloPsi}
 we get
 $(\Phi_{m_i}(x),\Phi_{m_r}(x))=\Z[x]$ for all $i=1,\dots, r-1$, so there exist $a_i(x),b_i(x)\in \Z[x]$ such that 
%As far as the isomorphism in the middle, nothing has to be proven for $r=1$. For $r\ge2$ we proceed by induction. 
%The base step $r=2$ follows from the CRT since $(\Phi_{m_1}(x),\Phi_{m_2}(x))=\Z[x]$ (see Corollary \ref{cor_nm}).
% For $r>2$ we have that $(\Phi_{m_i}(x),\Phi_{m_r}(x))=\Z[x]$ for all $i=1,\dots, r-1$, so there exist $a_i(x),b_i(x)\in \Z[x]$ such that 
 $$a_i(x)\Phi_{m_i}(x)+b_i(x)\Phi_{m_r}(x)=1.$$
 Multiplying the $r-1$ equations we  get 
  $$a(x)\Phi_{m_1}(x)\cdots\Phi_{m_{r-1}}(x)+b(x)\Phi_{m_r}(x)=1,$$
  for some $a(x),b(x)\in\Z[x]$, or equivalently
  $$(\Phi_{m_1}(x)\dots\Phi_{m_{r-1}}(x), \Phi_{m_r}(x))=\Z[x].$$
This ensures that  the CRT map 
  $$ \rho_r\colon \Z[x]/(\Phi_{m_1}(x)\dots\Phi_{m_r}(x))\cong\Z[x]/(\Phi_{m_1}(x)\dots\Phi_{m_{r-1}}(x))\times \Z[x]/(\Phi_{m_r}(x))$$
  is an isomorphism and  using equation \eqref{eq:psit} we can conclude that also $\psi_r$ is an isomorphism, proving (a).
     \end{proof}

We conclude this section with an arithmetical lemma which will be useful in Proposition \ref{gT}.
\begin{lemma}
\label{disuguaglianze}
{\sl
Let $q_1,\dots, q_k$ be pairwise distinct odd primes and let $ e_1,\dots e_k>0$. Then, for $\delta\ge1$
%\begin{equation}
% \label{eq-rank4}\phi(2^\delta q_1^{e_1}\!\cdots \!q_k^{e_k})=2^{\delta-1}\phi( q_1^{e_1}\!\cdots\! q_k^{e_k})\ge \sum_{i=1}^k\!2^{\delta-1}\phi( q_i^{e_i})=\sum_{i=1}^k\!\phi(2^{\delta} q_i^{e_i})
%  \end{equation} 
\begin{equation}
 \label{eq-rank3}\frac{\phi(2^\delta q_1^{e_1}\cdots q_k^{e_k})}2-1\ge \sum_{i=1}^k(\frac{\phi(2^{\delta} q_i^{e_i})}2-1)
 \end{equation}
and, if $\delta\ge2$,
 \begin{equation}
 \label{eq-rank2}\frac{\phi(2^\delta q_1^{e_1}\cdots q_k^{e_k})}2-1\ge \sum_{i=1}^k(\frac{\phi(2^{\delta-1} q_i^{e_i})}2-1)+\frac{\phi(2^{\delta})}2-1.
 \end{equation}
}
\end{lemma} 
\begin{proof}
Both inequalities are trivial for $k=0$, so let $k\ge1$. 
Since $\phi(q_i^{e_i})\ge2$   for all $i$,  the obvious relation $mn\ge m+n$ for all $m,n\ge2$, gives
  $$\phi(q_1^{e_1}\cdots q_k^{e_k})= \prod_{i=1}^k\phi( q_i^{e_i})\ge \sum_{i=1}^k\phi( q_i^{e_i}),
$$ 
from which we get
\begin{equation}
 \label{eq-rank4}\phi(2^\delta q_1^{e_1}\!\cdots \!q_k^{e_k})=2^{\delta-1}\phi( q_1^{e_1}\!\cdots\! q_k^{e_k})\ge \sum_{i=1}^k\!2^{\delta-1}\phi( q_i^{e_i})=\sum_{i=1}^k\!\phi(2^{\delta} q_i^{e_i})
  \end{equation} 
 and \eqref{eq-rank3} follows.
 
On the other hand, if $\delta\ge2$
%for 
%To  prove \eqref{eq-rank2} we note that 
%
%from \eqref{eq-rank4} we get  
%For $k=0$ the formula is an equality.  Let $k>0$; noting that
%since 
%$\delta\ge \epsilon+1$, we have 
$$\phi(2^\delta q_1^{e_1}\cdots q_k^{e_k}) =\phi(2^{\delta-1} q_1^{e_1}\cdots q_k^{e_k})+\phi(2^{\delta-1} q_1^{e_1}\cdots q_k^{e_k});$$
 using \eqref{eq-rank4} on both summands and then  the trivial estimate $\phi(2^{\delta-1} q_i^{e_i})\ge\phi(2^\delta)$, we get
$$\phi(2^\delta q_1^{e_1}\cdots q_k^{e_k})
 \ge\sum_{i=1}^k\phi(2^{{\delta-1}}q_i^{e_i})+ \sum_{i=1}^k\phi(2^{{\delta-1}}q_i^{e_i})\ge \sum_{i=1}^k\phi(2^{{\delta-1}}q_i^{e_i})+k\phi(2^\delta)$$
 and \eqref{eq-rank2} follows a fortiori.
\end{proof}

\section{Integral domains}
\label{integral-domains}
In this section we characterize the finitely generated groups which occur as group of units of an integral domain of any characteristic and in Proposition \ref{domini-int} those which are the group of units of integral extensions of $\Z$.

\begin{theo}
\label{domini-char0}
{\sl The  finitely generated abelian groups that occur as groups of units of integral domains of characteristic zero  are the groups of the form $C_{2n}\times\Z^g$, with $n\in\N$, $g\ge\frac{\phi(2n)}2-1$.
}
\end{theo}

\begin{proof} 
Suppose $A$ is an integral domain of characteristic zero whose group of units $A^*$ is finitely generated, so that  $A^*\cong T\times\Z^{g_A}$  where $T$ denotes the (finite) torsion subgroup. Let $K$ be the quotient field of $A$, then $T$ is a  finite multiplicative subgroup of $K^*$, hence it is a cyclic group. 

%\nota{mi pare che al posto di $A_0$ potrei mettere direttamente $\Z$ dato che ha caratteristica 0}
As noted in Section  \ref{notation}, the ring  $B={\Z}[T]$ has group of units isomorphic to $T\times \Z^{g_B}$ with $g_B\le g_A$.
% and $B^*$  differs from $A^*$ only by a, possibly trivial, power of $\Z$.  I
%it turns out that 
Hence, to prove that $A^*$ has the required form it is enough to restrict to the case  when $A=B$, namely it is finitely generated and integral over $\Z$. In this case,  its quotient field $K$ is a number field and  $A$ is an order of $K$.  By Dirichlet's Unit
Theorem $A^*\cong T\times\Z^{r+s-1}$  where $T$ is the (cyclic) group of roots of unity contained in $A$ and  $r$ and $2s$ are the number of real and non-real embeddings of $K$, respectively.
Clearly, $|T|$ is even since $-1\in A^*$. Let $T=\langle\zeta_{2n}\rangle$, then $\Z[\zeta_{2n}]\subseteq A$, so $ \Q(\zeta_{2n})\subseteq K$. For $n=1$ we have nothing to prove. If $n>1$, then all embeddings of $K$  in $\bar\Q$ must be non-real, so $r=0$ and $2s=[K:\Q]$. Since $\Q(\zeta_{2n})\subseteq K$ then $\frac{\phi(2n)}2 \mid s$ so the rank of $A^*$ is $g=s-1\ge\frac{\phi(2n)}2-1.$

As to the converse, let $n\ge1$ and 
let $K=\Q(\zeta_{2n})$. Then $\ok^*\cong C_{2n}\times\Z^{\frac{\phi(2n)}2-1}$ and  for any $k\ge1 $  the ring of Laurent polynomials in $k$ indeterminates $\ok[x_1,\dots,x_k,x_1^{-1},\dots, x_k^{-1}]$ has group of units isomorphic to $ C_{2n}\times\Z^{\frac{\phi(2n)}2-1+k}.$
\end{proof}

As a corollary we recover the characterization of the finite abelian groups which are  groups of units of an integral domain.
\begin{corollary}
\label{cor-domini}
{\sl
The finite abelian groups that occur as groups of units of integral domains of characteristic 0 are
the cyclic groups of order 2,4, or 6.
}
\end{corollary}
\begin{proof}
From Theorem \ref{domini-char0} we know that  
if $A$ is a domain such that $A^*$ is finitely generated, then $A^*\cong
%the finitely generated groups of units of a domain are of the form $
C_{2n}\times\Z^g$ with $g\ge (\frac{\phi(2n)}2-1)^*$, so we can have $g=0$ only for $n=1,2,3.$ 
\end{proof}

In Theorem \ref{domini-char0} we have seen that among the rings with finitely generated group of units and torsion subgroup isomorphic to  $C_{2n}$, the ring $A=\Z[\zeta_{2n}]$ has the minimum possible rank. The example of rings whose group of units has the same torsion subgroup, but  a greater rank are constructed in the theorem by localizing polynomial rings. In particular, the rings of our examples  are no longer integral over $\Z$. Actually, only some of these groups can also be obtained with units that are integral over $\Z$. The following proposition characterizes these cases.  

\begin{prop}
\label{domini-int}
{\sl 
The  finitely generated abelian groups that can be realized as group of units of an integral domain $A$, with $A$ integral over $\Z$, are
 the groups of the type $C_{2n}\times\Z^g$, with $n\ge1$, $g\ge0$ and $\phi(2n)\mid 2(g+1)$. 
}
\end{prop}
\begin{proof} 
%Up to replacing $A$ with $\Z[(A^*)_{tors}]$, 
%we can assume that 
%its  quotient field $K$  is a number field and  $A$ is an order of $K$.
%By Dirichlet's Unit
%Theorem,  $A^*\cong T\times\Z^{r+s-1}$  where $T=\langle\zeta_{2n}\rangle$ for some $n\ge1$ and  $[K:\Q]=r+2s$.
%It follows that $\Z[\zeta_{2n}]\subseteq A$, so $ \Q(\zeta_{2n})\subseteq K$. This shows that
%if $n>1$ all the embeddings of $K$  in $\bar\Q$ must be non-real ($r=0$) and  $[\Q(\zeta_{2n}):\Q]=\phi(2n) \mid 2s=[K:\Q]=2(g+1)$ where $g=s-1$ is the rank of $A^*$. For $n=1$ the divisibility condition is trivial.
%we can assume that 
%its  quotient field $K$  is a number field and  $A$ is an order of $K$.As noted in Section \ref{notation}  $A$ and $\Z[A^*]$ have the same group of units. Since is finitely generated o 
Up to replacing $A$ with $\Z[A^*]$ 
%(which has the same group of units and is finitely generated over $\Z$), 
we can assume that 
its  quotient field $K$  is a number field and that  $A$ is an order of $K$. Then the necessity of the condition follows from Theorem \ref{domini-char0} and from its proof, where it is shown that $\phi(2n)$ divides $2s = 2(g +1)$.

As for the converse, we have to construct examples of orders in number fields realizing  all the listed groups.
One  possible construction is the following.

For $n=1$ and $d\ge 1$, let $m$ be any integer such that $2d|\phi(m)$.
%Let $d,n\ge1$ and let $p$ be a prime such that 
%\begin{equation}
%\label{congrp}
%p\equiv1\pmod{2d}\, ;
%\end{equation}
%since there are infinitely many such primes (see for example  \cite[Corollary 2.11]{washington} or use  Dirichlet's Prime Number Theorem) we can assume $p\nmid n$. 
%%to be "big enough" ($p>dn$).
%
This condition guarantees that the field $\Q(\zeta_m+\zeta_m^{-1})$ contains a subfield
%inside the cyclotomic extension $\Q(\zeta_m)$ there is a unique subextension,
 $K_{d}$  of degree $d$ over $\Q$. Cleary, $K_d$ is totally real, so $r=d$, $s=0$ and the only roots of unity in  $K_{d}$ are $\pm1$,
%Now, $d\mid\frac{p-1}2$ hence  $K_{d,p}\subseteq\Q(\zeta_p+\zeta_p^{-1})$ and $K_{d,p}$ is totally real; hence $r=d$, $s=0$ and the only roots of unity in  $K_{p,d}$ are $\pm1$. 
hence 
the group of units of the integers of $K_{d}$ is isomorphic to $C_2\times \Z^{d-1}$. 
%This family gives the examples for $n=1$.
%Moreover, the only prime which ramify in $K_{d,p}$ is $p$.
%The extension $\Q(\zeta_p)/\Q$ is cyclic and the only ramified prime is $p$ which is totally ramified: all these properties are preserved when considering subextensions, so they hold for $K_{d,p}$. 
%Moreover, $d\mid\frac{p-1}2$, so  $K_{d,p}\subseteq\Q(\zeta_p+\zeta_p^{-1})$, namely $K_{d,p}$ is real. It follows that  units of the integers of  $K_{d,p}$ have rank $d-1$ and are $\pm1$ 
%Finally, the result on the group of units, follows from Theorem \ref{dirichlet}, since  and, being $K_{d,p}$ real, the only root of unity contained in it are $\pm1.$

 Consider now the case $n>1$.
Let $d\ge1$ and let $p$ be a prime such that 
\begin{equation}
\label{congrp}
p\equiv1\pmod{2d}\, ;
\end{equation}
since there are infinitely many such primes (see for example  \cite[Corollary 2.11]{washington} or use  Dirichlet's Prime Number Theorem) we can assume $p\nmid n$. 
%to be "big enough" ($p>dn$).
The congruence condition guarantees that inside the cyclotomic extension $\Q(\zeta_p)$ there is a (unique) subextension, $K_{d,p}$, of degree $d$ over $\Q$, which is indeed contained in the real subfield $\Q(\zeta_p+\zeta_p^{-1})$.
%Now, $d\mid\frac{p-1}2$ hence  $K_{d,p}\subseteq\Q(\zeta_p+\zeta_p^{-1})$ and $K_{d,p}$ is totally real; hence $r=d$, $s=0$ and the only roots of units in  $K_{p,d}$ are $\pm1$. This shows that 
%the group of units of the integers of $K_{d,p}$ is isomorphic to $C_2\times \Z^{d-1}$. 
 %For each integer $d\ge1$,  let $p$ be a prime satisfying equation \eqref{congrp} which we 
%assume to be ``big enough" ($p>dn$).  
Put $L=L_{d,p,n}=K_{d,p}\Q(\zeta_{2n})$ and denote by ${\mathcal O}_L={\mathcal O}_{L_{d,p,n}}$ its ring of integers. We claim that
 $${\mathcal O}_L^*\cong C_{2n}\times\Z^{\frac{d\phi(2n)}2-1}.$$
% whereas for $n=1$
% $${\mathcal O}_L^*\cong C_{2}\times\Z^{d-1}.$$
 In fact,  $({\mathcal O}_L^*)_{tors}=\langle \zeta_{2n}\rangle$ since $\zeta_{2n}\in{\mathcal O}_L^*$, $\zeta_p\not\in{\mathcal O}_L^*$.  
 
% In fact, $\Q(\zeta_{2n})\subset L\subset \Q(\zeta_p)\Q(\zeta_{2n})=\Q(\zeta_{2np})$ ($p\nmid n$)
% %, but,  since $p$ is big, 
% and a degree argument shows that $\zeta_p$ can not belong to $L,$ proving that the only roots of unity in ${\mathcal O}_L^*$ are the powers of $\zeta_{2n}$.

To compute the rank of ${\mathcal O}_L^*$, we note that 
$\Q(\zeta_p)$ is arithmetically  disjoint from $\Q(\zeta_{2n})$ since $(p,2n)=1$, hence also 
$K_{d,p}$ is arithmetically  disjoint from $\Q(\zeta_{2n})$  and $[L:\Q]=[K_{d,p}:\Q][\Q(\zeta_{2n}):\Q]=d\phi(2n)$. Moreover, $L$ is Galois over $\Q$ and   all its embeddings are non-real, so the rank of its group of units is $s-1=\frac{d\phi(2n)}2-1$.

\end{proof}

To complete the description of the finitely generated groups of units of integral domains, in the following theorem we present the simple result for finite characteristic rings.

\begin{theo}
\label{domini-charp}
{\sl The  finitely generated abelian groups that occur as groups of units of an integral domain of characteristic $p$ are the groups  of the form $\F_{p^n}^*\times\Z^g$ with $n\ge1$ and $g\ge0$.
}
\end{theo}
\begin{proof} 
Let $A$ be a domain 
and let $A^*\cong (A^*)_{tors}\times \Z^g$  with $ (A^*)_{tors}$ finite and $g\ge0$.
By Lemma \ref{minimum}, for  $B=\F_p[ (A^*)_{tors}]$ we have $B^*=(A^*)_{tors}$. 
Now, $B$ is a finite integral domain (it is a  finitely generated integral extension of $\F_p$), whence it is a finite field, namely $B\cong\F_{p^n}$ for some $n\ge1$. It follows that  $(A^*)_{tors}=B^*\cong\F_{p^n}^*$, and $A^*\cong  \F_{p^n}^*\times\Z^g$ as required. 

Conversely, for $n\ge1$ and $g\ge0$, the group $\F_{p^n}^*\times\Z^g$ is isomorphic to the group of units of the ring of  Laurent polynomials with coefficients in $\F_{p^n}$ and $g$ indeterminates.
\end{proof}

\section{Torsion-free rings: preliminary results}
\label{sec:torsion-free-prel} 
A commutative ring $A$ is called  torsion-free if its  only element of finite  additive  order is 0. Clearly, a torsion-free ring has characteristic zero.

For a torsion-free ring $A$ we put $Q_A=A\otimes_\Z \Q$. We note that  in this case the map
$$\iota\colon A\to Q_A$$ 
defined by $a\mapsto a\otimes 1$ is an embedding, so we will say that $A\subseteq Q_A$.

As noted in Section \ref{notation} (Lemma \ref{minimum} and Remark \ref{rem-differenze}), to characterize the  finitely generated abelian groups $T\times\Z^g$ that  arise as groups of units  of torsion-free rings, a substantial step is the study of  the subrings that are generated  over $\Z$  by  units of finite order. In fact,  in this subclass all possible torsion subgroups $T$ are realized and, for each $T$, the minimum possible rank $g(T)$ is attained. 
This case is much easier  to study since if $A$ is integral over $\Z$ then $Q_A$ is a finite dimensional $\Q$-algebra and $A$ is an order of $Q_A$. 
In this  section and in the first part of the next one we will restrict  to this case; then it will be easy to deal with the general case.

\smallskip

The following  lemma allows us to describe the ring $A$ when it is generated by one torsion unit and it is a generalization  of  \cite[Lemma 4.2]{dcdBLMS}.
%, is the fundamental step of our proof; in this subsection we only use it in the less general form which already appeared in \cite[Lemma 4.2]{dcdBLMS}, namely for $K=\Q$. However,  we state and prove it  in the more general form we will need in the following.

 \begin{lemma}
\label{lemma1}
\label{lemma1bis}
\label{lemmadx}
 {\sl  Let $K$ be a number field and let $\ok$ be its ring of integers. Assume that $\ok\subseteq A$. 
 Let  $\alpha\in A^*$ be an element of order $n$, let 
$$\varphi_\alpha\colon \ok[x]\to A$$
 be the evaluation homomorphism
  $p(x)\mapsto p(\alpha)$.
  
Then  $\ker(\varphi_\alpha) =(\mu_\alpha(x))$ with
 $$\mu_\alpha(x)=\Psi_{m_1}(x)\cdots\Psi_{m_r}(x)$$
   where, for each $i$, $\Psi_{m_i}(x)\in\ok[x] $ denotes the minimal polynomial  over $K$ of a primitive $m_i$-th root of unity. Moreover,  the $\Psi_{m_i}(x)$'s are pairwise distinct and
$[m_1,\dots,m_r]=lcm\{m_1,\dots,m_r\}=n$.
}
 \end{lemma}
 \begin{proof}
The element $\alpha$ has order $n$, so $x^n-1\in \ker(\varphi_\alpha)$. 
Denote by $\tilde\varphi_\alpha\colon K[x]\to {Q_A}$ the extension of $\varphi_\alpha$.
Then,  there exists a monic polynomial $\mu_\alpha(x)\in K[x]$ such that $\ker(\tilde \varphi_\alpha)=(\mu_\alpha(x))$.
Clearly, $\mu_\alpha(\alpha)=0$ and  $\mu_\alpha(x)$ divides  the separable polynomial  $x^n-1$ in $K[x]$,
 $$\mu_\alpha(x)\mid (x^{n}-1)=\prod_{m|n}\Phi_{m}(x).$$
Now,  each $\Phi_m$ factors as a product of distinct cyclotomic polynomials over $K$, hence  $\mu_\alpha(x)$ factors in $K[x]$ as
$$\mu_\alpha(x)=\Psi_{m_1}(x)\cdots\Psi_{m_r}(x),$$ 
where $\Psi_{m_i}(x) $ denotes the minimal polynomial  over $K$ of a primitive $m_i$-th root of unity. The $\Psi_{m_i}(x)$'s are pairwise distinct since $x^n-1$ is separable; moreover, $\Psi_{m_i}(x)\in \ok[x]$ for all $i$, so 
 $\mu_\alpha(x)\in\ok[x]$ and $\mu_\alpha(x)\in \ker(\varphi_\alpha)$.
 
On the other hand, for each $f(x)\in\ker(\varphi_\alpha)$ we have $\mu_\alpha(x)|f(x)$ in $K[x]$ and since  $\mu_\alpha(x)\in\ok[x]$ is a monic polynomial, then it divides $f(x)$ in $\ok[x]$. This proves that $\ker(\varphi_\alpha)=(\mu_\alpha(x))$.

Let $[m_1,\dots,m_r]=m$.  Since $m_i\vert n$ for all $i$, then  $m\mid n$. In fact $m=n$, since  otherwise $\mu_\alpha(x)\mid x^{m}-1$
and therefore  $\alpha^{m}=1$, contrary to our assumption.
\end{proof}

\begin{prop}
\label{QB}
{\sl Let $A=\Z[\alpha_1,\dots,\alpha_s]$, where, for all $i$, $\alpha_i$ is a unit of finite order and assume that $A$ is torsion free. 
Then the $\Q$-algebra $Q_A=A\otimes_\Z\Q$ is a finite direct product of cyclotomic fields.
In particular, $Q_A$ is a semisimple $\Q$-algebra. 
}
\end{prop}

\begin{proof} 
 For $\alpha=\alpha_i$, in the notation of Lemma \ref{lemma1}, let $\ker(\varphi_\alpha)=(\mu_\alpha(x))$ and assume
 $$\mu_\alpha(x)=\Phi_{m_1}(x)\cdots\Phi_{m_r}(x)$$  
for some distinct $m_1,\dots m_r$.
Then the CRT gives 
   $$\Q[\alpha]=\Z[\alpha]\otimes_\Z\Q\cong \Q[x]/(\mu_\alpha(x))\cong \prod_{i=1}^r\Q[x]/(\Phi_{m_i}(x))\cong \prod_{i=1}^r\Q(\zeta_{m_i}).$$ 
%Taking into account that 
  Now, the degree of $\zeta_m$ over $ \Q(\zeta_n)$ is $\phi(m)/\phi((n,m))$, so $m$-th cyclotomic polynomial $\Phi_m(x)$ splits  into $\phi((n,m))$ of factors in $ \Q(\zeta_n)[x]$. It follows that
   $$\Q(\zeta_n)\otimes_\Q\Q(\zeta_m)\cong \Q(\zeta_n)[x]/(\Phi_m(x))\cong \Q(\zeta_{[n,m]})^{\phi((n,m))},$$
so the $\Q$-algebra 
   ${\mathcal Q}=\Q[\alpha_1]\otimes_\Q\cdots \otimes_\Q\Q[\alpha_s]$ is a product of cyclotomic fields. It turn out that the same is true for   $Q_A=\Q[ \alpha_1,\dots, \alpha_s]$  since it is the epimorphic image of  ${\mathcal Q}$ via the $\Q$-algebra homomorphism defined by $\alpha_1\otimes\cdots \otimes\alpha_s\mapsto \alpha_1\cdots\alpha_s$.
\end{proof}

\begin{remark}
\label{maxord}
{\rm The last proposition shows that  the $\Q$-algebra $Q_A$ is isomorphic to $\prod_{i=1}^t\Q(\zeta_{n_i})$ for some $n_1,\dots, n_t$, namely, it is 
  semisimple and of finite dimension over the perfect field $\Q$, hence  it is separable (see for example \cite[Cor. 7.6]{CurtisReiner1981-1}). Moreover, $Q_A$ is clearly commutative, so by   \cite[Prop. 26.10]{CurtisReiner1981-1} it has a unique maximal order  $\cm_A$, which is  the integral closure of $\Z$ in $Q_A$, namely 
$$\cm_A\cong\prod_{i=1}^t\Z[\zeta_{n_i}].$$ 
Since $A$ is an order of $Q_A$, then  $A$ is a subring of $\cm_A$, therefore the rings we are  taking into account are subrings of finite products of cyclotomic rings.
}
\end{remark}
The next lemma shows that the groups of units of all  orders of $Q_A$ have the same rank (see also \cite[Prop. 2.5]{Sehgal1993} or \cite[Lemma 3.7]{bartel_lenstra_2017}). 

\begin{lemma}
 \label{finito}
 {\sl Let  $R$ be an order  of a commutative and finitely generated semisimple $\Q$-algebra $Q$ and let $\M$ denote its maximal order. Then $R^*$ has the same rank as $\M^*$.}
 \end{lemma}
   \begin{proof}
Each  order $R$ of $Q$ is a subring of finite index of $\M$, since both are $\Z$-modules of the  same finite rank. 
 Let $[\M:R]=m$, then the ideal $m\M$ is contained in $R$ and $\M/m\M$ is a finite ring.
 
 Consider  the projection $\pi\colon \M\to \M/m\M$. Since $\pi$ is a ring homomorphism,  it sends the unit of $\M$ into the  unit of the quotient and the restriction of $\pi\colon\M^*\to( \M/m\M)^*$ is a group homomorphism. 
 
 Let $|( \M/m\M)^*|=c$. For each $\epsilon\in \M^*$ we have that $\epsilon^c\equiv1\pmod{m\M}$ so $\epsilon^c-1\in{m\M}\subset R$. Now, $R$ and $\M$ have the same identity, hence $\epsilon^c\in R$ and $(\M^*)^c\subseteq R^*\subseteq \M^*$.  Finally, since $(\M^*)^c$ and $\M^*$ have the same rank,  this is  also the rank of $R^*.$
  \end{proof}

\begin{corollary}
\label{unitsM}
{\sl
In the notation of Proposition \ref{QB}, let $Q_A=\prod_{i=1}^t\Q(\zeta_{n_i})$. Then,  $A^*\cong T\times\Z^g$ where 
$$g=\sum_{i=1}^t(\frac{\phi(n_i)}2-1)^*$$and $T$ is a subgroup of even order of $U=\prod_{i=1}^t\langle-\zeta_{n_i}\rangle$.}
\end{corollary}
\begin{proof}
The order $A$ is contained in the maximal order $\cm_A\cong\prod_{i=1}^t\Z[\zeta_{n_i}]$, hence $\{\pm1\}<A^*<
\cm_A^*$ and by Lemma \ref{finito} the two groups have the same rank $g$. The result follows since
$$\cm_A^*\cong \prod_{i=1}^t\Z[\zeta_{n_i}]^*\cong\prod_{i=1}^t\left(\langle-\zeta_{n_i}\rangle\times\Z^{(\frac{\phi(n_i)}2-1)^*}\right)\cong U\times\Z^g.$$
\end{proof}

The next proposition classifies the cases when  $\Z[\alpha]$ coincides with $\cm$.

\begin{prop}
\label{prop:zetaalpha}
{\sl
Let $\alpha\in A^*$ be an element of finite order. Denote by  $\varphi_\alpha\colon\Z[x]\to A$  the evaluation homomorphism and let $\mu_\alpha(x)=\Phi_{m_1}(x)\dots\Phi_{m_r}(x)$ be a generator of $\ker(\varphi_\alpha)$. 
Then 
$$
  \Z[\alpha]\cong\prod_{i=1}^{r}\Z[\zeta_{m_i}]
$$
if and only if , for all $i,j$, the ratio $m_i/m_j$ is not a prime power. 

In this case   
$$\Z[\alpha]^*\cong \prod_{i=1}^{r}\langle-\zeta_{m_i}\rangle\times\Z^{\sum_{i=1}^{r}(\frac{\phi(m_i)}2-1)^*}.$$
}
\end{prop}
\begin{proof}
%The question of when  $\Z[\alpha]$  can be  read in term of the 
%the embedding of $\Z[\alpha]$ into the maximal order $\cm$ of $Q_{\Z[\alpha]}$ can be studied via the 
%CRT map, via  
Consider the following commutative diagram, where the vertical arrows are the obvious isomorphisms
\begin{equation}
\label{commdiag}
\begin{tikzcd}
 \Z[\alpha] \arrow[r, hookrightarrow]& \cm=\prod\limits_{i=1}^r\Z[\zeta_{m_i}]  \\
{\Z[x]}/{( \prod\limits_{i=1}^r\Phi_{m_i}(x))}\arrow[u , "\rotcong"] \arrow[r, hookrightarrow, "\psi"]
 & \prod\limits_{i=1}^r{\Z[x]}/{(\Phi_{m_i}(x))} \arrow[u,"\rotcong"]
\end{tikzcd}
\end{equation}
The diagram shows that $\Z[\alpha]=\cm$ if and only if the CRT map is onto and this is classified in
%The equivalence follows from the diagram \eqref{commdiag} and 
Proposition \ref{prop:cinese}. The description of $\Z[\alpha]^*$ follows immediately.
%We have the following chain of isomorphisms:
% $$
%  \Z[\alpha]\cong \Z[x]/(\Phi_{m_1}(x)\dots\Phi_{m_{r}}(x))]\cong \prod_{i=1}^{r}\Z[x]/(\Phi_{m_i}(x))\cong\prod_{i=1}^{r}\Z[\zeta_{m_i}].
%  $$
% \new{ The first and the last  are those induced by the substitution homomorphisms.} 
%For the isomorphism in the middle, we proceed by induction. Nothing has to be proven for $r=1$. For $r\ge2$ we note that, since $m_i/m_r$ is not a power of a prime, by Corollary \ref{cor_nm} we get that 
% $(\Phi_{m_i}(x),\Phi_{m_r}(x))=\Z[x]$ for all $i=1,\dots, r-1$, so there exist $a_i(x),b_i(x)\in \Z[x]$ such that 
%
%%As far as the isomorphism in the middle, nothing has to be proven for $r=1$. For $r\ge2$ we proceed by induction. 
%%The base step $r=2$ follows from the CRT since $(\Phi_{m_1}(x),\Phi_{m_2}(x))=\Z[x]$ (see Corollary \ref{cor_nm}).
%% For $r>2$ we have that $(\Phi_{m_i}(x),\Phi_{m_r}(x))=\Z[x]$ for all $i=1,\dots, r-1$, so there exist $a_i(x),b_i(x)\in \Z[x]$ such that 
% $$a_i(x)\Phi_{m_i}(x)+b_i(x)\Phi_{m_r}(x)=1.$$
% Multiplying the $r-1$ equations we  get 
%  $$a(x)\Phi_{m_1}(x)\cdots\Phi_{m_{r-1}}(x)+b(x)\Phi_{m_r}(x)=1,$$
%  for some $a(x),b(x)\in\Z[x]$, hence the CRT gives 
%  $$ \Z[x]/(\Phi_{m_1}(x)\dots\Phi_{m_r}(x))\cong\Z[x]/(\Phi_{m_1}(x)\dots\Phi_{m_{r-1}}(x))\times \Z[x]/(\Phi_{m_r}(x))$$
%  and we can conclude by induction.
%  
%\new{The computation of the group of units descends.}
\end{proof}

%Let following  two easy examples.

\begin{example}
{\rm
Let $\cm=\Z[\zeta_3]\times\Z[i]$ and let $\alpha=(\zeta_3, i)\in\cm$. The element $\alpha$ is a unit of order 12, $\mu_\alpha(x)=\Phi_3(x)\Phi_4(x)$ and 
%. Let $\varphi_\alpha\colon\Z[x]\to \cm$ the substitution homomorphism sending
%$\varphi_\alpha(1)=(1,1)$ and $\varphi_\alpha(x)=\alpha=(\zeta_3,i)$. We have 
%$\ker(\varphi_\alpha)=(\Phi_3(x)\Phi_4(x))$ and 
$\Z[\alpha]\cong\Z[x]/(\Phi_3(x)\Phi_4(x))$. By last proposition
%Proposition \ref{prop:cinese} says that that the CRT map in this case is an isomorphism, so
$\Z[\alpha]\cong \cm$ and $(\Z[\alpha])^*_{tors}=(\cm^*)_{tors}\cong C_6\times C_4$.
%where the second isomorphism is given by the  Chinese Remainder Theorem
%    since   $ (\Phi_3(x),\Phi_4(x))=\Z[x]$ (see \old{Proposition {\ref{cycloPsi}}} \new{Corollary \ref{cor_nm}}).
}
\end{example}
\begin{example}
{\rm
Let  $\cm=\Z[\zeta_3]\times\Z[\zeta_9]$ and let $\alpha=(\zeta_3, \zeta_9)\in\cm$. Clearly, $\alpha$ is a unit of order 9 and  $\Z[\alpha]\cong\Z[x]/(\Phi_3(x)\Phi_9(x))$. Proposition \ref{prop:zetaalpha} shows that $\Z[\alpha]\subsetneq \cm$ and it is easy to see  that $(\Z[\alpha])^*_{tors}\cong C_9$, in fact  $(\zeta_3,1)\not\in \Z[\alpha]$.
}
\end{example}

In the following proposition we compute the groups of units of torsion free rings of a particular form which will be useful  in the next section.   Actually, using the results of this section together with those of \S \ref{subs:cyclo} one could prove more general results, substantially  with the same methods, but this would require a greater technical effort. However, this is beyond our scope, so we decided to limit the generality to what is necessary for our application.

\begin{prop}
\label{units_1n}
{\sl
Let $p$ be a  prime and let $l$ be a positive even integer such that $l=l_1p^b$ with $(l_1, p)=1$. Let $a>b$ and let  $\Psi_{p^a, p^b}(x)$ denote the minimal polynomial of $\zeta_{p^a}$ over $\Z[\zeta_{p^b}]$.\footnote[2]{$\Psi_{p^a, p^b}(x)$ is also  the minimal polynomial of $\zeta_{p^a}$ over $\Z[\zeta_{l}]$}
Then 
$$\left(\frac{\Z[\zeta_l][x]}{((x-1)\Psi_{p^a, p^b}(x))}\right)^*\cong C_l\times C_{p^a}\times \Z^g$$ 
where $g=(\frac{\phi(l)}2-1)^*+(\frac{\phi(l_1p^{a})}2-1).$

%Let $p$ be a  prime and let $l$ be a positive even integer such that $l=l_ip^e$ with $(l_1, p)=1$. Let $a>e$ and let $
%\alpha=(1,\zeta_{p^a})\in \cm=\Z[\zeta_l]\times \Z[\zeta_{l_1p^a}]$. Then 
%$$\Z[\zeta_l][\alpha]^*\cong C_l\times C_{p^a}\times \Z^g$$ 
%where $g=(\frac{\phi(l)}2-1)^*+(\frac{\phi(l_1p^{a})}2-1).$
%and let $K=\Q(\zeta_l)$.
%% and let $\ok$ be its ring of integers.
%
%Let $p$ be a  prime, let $a>0$ and assume that $l=l_1p^b$  with $(l_1, p)=1$ and $b\le a$.
% 
%Denote by $\Psi_{p^a, p^b}(x)$ the minimal polynomial of $\zeta_{p^a}$ over $K$. Then 
%$$(\Z[\zeta_l][x]/((x-1)\Psi_{p^a, p^b}(x)))^*\cong C_l\times C_{p^a}\times \Z^g$$ 
%where $g=(\frac{\phi(l)}2-1)+(\frac{\phi(l_1p^{a})}2-1).$
}\end{prop}
\begin{proof}
%First of all we note that $\Z[\zeta_l][\alpha]\cong\Z[\zeta_l][x]/((x-1)\Psi_{p^a, p^b}(x))$, so it is an order of 
%$\Q(\zeta_l)\times\Q(\zeta_{lp^a})$ and 
%$$\rank(\Z[\zeta_l][\alpha] ^*)=(\frac{\phi(l)}2-1)^*+(\frac{\phi(l_1p^{a})}2-1).$$ 
The  ring $\Z[\zeta_l][x]/((x-1)\Psi_{p^a, p^b}(x))$ embeds into the maximal order $\cm=\Z[\zeta_l]\times\Z[\zeta_l][\zeta_{p^a}]\cong\Z[\zeta_l]\times\Z[\zeta_{l_1p^{a}}]$ via the CRT map:
$$\psi\colon\Z[\zeta_l][x]/((x-1)\Psi_{p^a, p^b}(x))\to\Z[\zeta_l]\times \Z[\zeta_l][x]/(\Psi_{p^a, p^b}(x))\cong\cm,$$
 then 
$$\rank\left(\left(\frac{\Z[\zeta_l][x]}{((x-1)\Psi_{p^a, p^b}(x))}\right)^*\right)=(\frac{\phi(l)}2-1)^*+(\frac{\phi(l_1p^{a})}2-1).$$

%Let $\varphi\colon\Z[\zeta_l][\alpha]\to\cm$ be the embedding defined by $1\mapsto (1,1)$  and $x\mapsto\alpha$: via this identification  $\varphi(\Z[\zeta_l][\alpha])$ is an order of $\Q(\zeta_l)\times\Q(\zeta_{lp^a})$, hence
%$$\rank(\Z[\zeta_l][\alpha] ^*)=(\frac{\phi(l)}2-1)^*+(\frac{\phi(l_1p^{a})}2-1).$$
As for the torsion units, let
$$T=\psi\left(\left(\frac{\Z[\zeta_l][x]}{((x-1)\Psi_{p^a, p^b}(x))}\right)^*_{tors}\right).$$ 
Clearly, $T$ is the subgroup of 
$U=\langle\zeta_{l}\rangle\times\langle\zeta_{l_1p^{a}}\rangle\cong C_{l}\times C_{l_1p^{a}}$ made by the units belonging to ${\rm Im}(\psi)=\{(a(1),a(\zeta_{p^a}))\mid a(x)\in\Z[\zeta_l][x]\}$.
We will show that all of them are 
  {\em trivial units}, namely they belong to the subgroup $T_0$ generated by $\psi(\zeta_l)=(\zeta_l,\zeta_l)$ and $\psi(x)=(1,\zeta_{p^a})$.
We note that $T_0\cong C_l\times C_{p^a}$,\ since  
$\langle(\zeta_l,\zeta_l)\rangle\cap\langle (1,\zeta_{p^a})\rangle=(1,1)$.

Let $u=(\zeta_l^i, \zeta_l^j\zeta_{p^a}^k)\in U$, then  $u$ is equivalent to $v=(\zeta_l^{i-j},1)$ modulo $T_0$,
so $u\in T$ if and only if  $v-(1,1)=(\zeta_l^{i-j}-1, 0)\in{\rm Im}(\psi)$.

This means that there exists 
$a(x)\in\Z[\zeta_l][x]$ such that  
$$\zeta_l^{i-j} -1=a(1)\Psi_{p^a,p^b}(1).$$
By  Lemma \ref{Psin1}, $(\Psi_{p^a,p^b}(1))= P_b$ where $P_b=(1-\zeta_{p^b})$ if $b\ge1$ and $P_0=(p)$, hence  last equation implies
\begin{equation}
\label{zetai-j}
\zeta_l^{i-j} -1\in P_b.
\end{equation}

Let  $\nu=l/(l, i-j)$, then 
%The element $\zeta_l^{i-j}$is a primitive $\nu$-th root of unity where $\nu=l/(l, i-j)$  and 
 \eqref{zetai-j} can be rewritten as $\zeta_\nu -1\in P_b$ and, using Lemma \ref{lemmawashington}, we get that this holds  if and only if  $\nu\vert p^b$.

If $b\ge1$, $\nu\mid p^b$ exactly when  $i\equiv j\pmod{l_1}$.  Let $j=i+hl_1$, then $u=(\zeta_l^i,\zeta_l^i\zeta_{p^b}^h\zeta_{p^a}^k)$ and clearly this element is  in $T_0$.

If $b=0$ equation \eqref{zetai-j} can hold only for $p=2$, so $\nu=1$ or $2$ and $i\equiv j\pmod{l_12^{b-1}}$. Letting $j=i+tl_12^{b-1}$ ($t=0,1$) the unit $u=(\zeta_l^i,(-1)^t\zeta_l^i\zeta_{2^a}^k)=(\zeta_l^i,\zeta_l^i\zeta_{2^a}^{k+t2^{a-1}})$ and  clearly it belongs to $T_0$. 

This proves that  $T=T_0$ and hence it has the required decomposition.
\end{proof}

\section{Torsion-free rings: the classification theorem}
\label{sec:torfree}

 Our aim is to classify the abelian and finitely generated  groups which arise as groups of units of  torsion-free rings. This question is twofold: on the one hand, we have to establish which finite groups $T$  (up to isomorphism) can be the torsion subgroup of $A^*$ when $A$ is a torsion-free ring. On the other hand, we have to determine  the possible values of the rank, $g(A)$, of $A^*$ when $(A^*)_{tors}\cong T$. Theorem \ref{torfree} gives a complete answer to both questions.

Let $T$ be a  finite abelian group of even order. In this section we will use the following notation  for the decomposition of $T$ as a product of cyclic factors that we fix once and for all. We will refer to this notation as to the ``standard'' notation for $T$, or we will call \eqref{eqT} the ``standard'' decomposition of $T$

\smallskip

{\bf Standard notation for $T$.} Let $\epsilon=\epsilon(T)$  be the minimum exponent of 2 in the
decomposition of  $T$  as direct sum of cyclic groups.
Then $T$ can be uniquely written as 
\begin{equation}
\label{eqT}
T\cong\prod_{i=1}^s C_{p_i^{a_i}}\times\prod_{\iota=1}^\rho C_{2^{\epsilon_\iota}}\times C_{2^{\epsilon}}^\sigma
\end{equation}
where $s,\rho\ge0$, $\sigma\ge1$ and\\
- for all $i=1,\dots,s$ the $p_i$'s are odd prime numbers not necessarily distinct and $a_i\ge1$;\\
- $\epsilon=\epsilon(T)\ge1$ and $\epsilon_\iota>\epsilon$ for all $\iota=1,\dots,\rho$.\\
Assume that $p_1,\dots, p_{s_0}$ are the distinct primes in the set $\{p_1,\dots, p_s\}$. Denoting by $T_{p_i}$ the $p_i$-Sylow of $T$, for $i=1,\dots, s_0$, and by $T_2$ its 2-Sylow, we can also write $T$ as 
\begin{equation}
\label{eqTp}
T\cong \prod_{i=1}^{s_0}T_{p_i}\times T_2.
\end{equation}
\smallskip

 As usual, we call the decomposition in \eqref{eqTp} the Sylow decomposition.

\begin{theo}
\label{torfree}
{\sl
 Let $T$ be a finite abelian group of even order. Referring to the ``standard" notation for $T$, we define
\begin{equation}
\label{eqgT}g(T)=\sum_{i=1}^s(\frac{\phi(2^{\epsilon} p_i^{a_i})}2-1)+ \sum_{\iota=1}^\rho(\frac{\phi(2^{\epsilon_\iota})}2-1)+  c(T)\end{equation}
where 
$$c(T)=\begin{cases}({\sigma-s})(\frac{\phi(2^{\epsilon} )}2-1)^*& {\rm for}\ s<\sigma\\
0 & {\rm for }\ s_0\le \sigma\le s\\
(\frac{\phi(2^{\epsilon} )}2-1)^*& {\rm for }\ \sigma<s_0.
\end{cases}$$

Then there exists a torsion free ring $A$ with
$$A^*\cong T\times\Z^r$$
 if and only if  $r\ge g(T).$
}
\end{theo}
As a particular case of this theorem we re-obtain the classification of finite groups which occur as groups of units of torsion-free rings, already found in \cite[Thm 4.1]{dcdBLMS}.
\begin{corollary}
\label{cor-torsionfree}
{\sl
The finite abelian groups which are the groups of units of  {torsion-free} rings are all those of the form
$$C_2^a\times C_4^b\times C_3^c$$
where $a,b,c\in\N$, $a+b\ge1$ and $a\ge1$ if $c\ge1.$
}
\end{corollary}
\begin{proof}
A finite abelian group $T$ of even order  is the group of units of a torsion-free ring if and only if $g(T)=0$. In the ``standard" notation for $T$, this means that $\frac{\phi(2^{\epsilon} p_i^{a_i})}2-1=0$ for all $i=1,\dots,s$, $\frac{\phi(2^{\epsilon_\iota})}2-1=0$ for each $\iota=1,\dots,\rho$ and $c(T)=0$.
If  $s=0$ this gives $\epsilon=1$ and $\epsilon_\iota\le2$ for all $\iota$,  or $\epsilon=2$ and $\rho=0 $.
If $s>0$, then $p_i=3$ for all $i$, $\epsilon=1$ and $\epsilon_\iota\le2$ for all $\iota$.
\end{proof}
\smallskip

Before proceeding with the proof we point out that all the difficulties relative  to the realization of a group $T$ come from its 2-torsion part.
%
%
%to point out some  phenomenon which at first sight seems paradoxical: sometimes for a groups $T$ with subgroup $T'$ one can have that $g(T)<g(T')$.
%
%to give an idea of why the proof of Theorem  \ref{torfree} can not be simple. The complications in the realization of some torsion group $T$ arise in relation on the $2$-torsion part of $T$. 
The following examples show a phenomenon which at first sight may seem paradoxical: it may happen that a group $T$ has a subgroup $T'$ for which $g(T)<g(T')$.
\begin{example}
\label{ordine80}
{\rm 
Let $T=C_2\times C_8\times C_5$. In this case $\epsilon=1$ and $g(T)=2$: in fact, choosing $A$ equal to  the maximal order $\cm=\Z[\zeta_8]\times\Z[\zeta_5]$  we have $A^*\cong T\times\Z^2$.
}
\end{example}
\begin{example}
\label{ordine40}
{\rm
Let  $T\cong C_8\times C_5$ and let $A$ be  a torsion-free ring such that $(A^*)_{tors}\cong T$. Then,  $A$ contains a unit $\alpha$ of order 8 and a unit $\beta$ of order 5.  
Then in the notation of Lemma \ref{lemma1}, we have that $\Phi_5(x)\mid\mu_\alpha(x)$ and   $\Phi_8(x)\mid\mu_\beta(x)$, so  ${\mathcal M}$, the  maximal order of $A$, 
%using Lemma \ref{lemma1}, 
 must contain a direct factor with a subring isomorphic to  $\Z[\zeta_{8}]$ and one which contains $\Z[\zeta_{5}]$. There are two minimal  possibilities: ${\mathcal M}=\Z[\zeta_{8}]\times\Z[\zeta_{5}]$ or ${\mathcal M}=\Z[\zeta_{40}]$. The first possibility has to be excluded since each  order of a maximal order containing  $\Z[\zeta_{8}]\times\Z[\zeta_{5}]$ has at least 3 units of order 2 (this will be clear after Lemma   \ref{epsilon}).
In this case Theorem \ref{torfree} shows that $g(T)= \phi(40)/2-1=7$. 
}
\end{example}

The proof of Theorem \ref{torfree} is quite long. For the convenience of the reader, we  separate the ``only if'' part and the ``if" part.
Both parts require a number of auxiliary results that we will prove separately, in order to make it easier to follow the main argument.

\subsection{Proof of Theorem \ref{torfree}: the ``only if" part}

Let  $A$ be a torsion free ring with finitely generated group of units, such that $(A^*)_{tors}\cong T$. We have to prove the $\rank(A^*)\ge g(T)$.

To this aim,  by Lemma \ref{minimum}, we can assume that   $A=\Z[(A^*)_{tors}]$ and 
Proposition \ref{QB} says that there exist $n_1,\dots,n_t$ such that $Q_A=A\otimes_\Z\Q\cong\prod_{j=1}^t\Q(\zeta_{n_j})$. Now, by Lemma \ref{finito}, the rank of $A^*$ is equal to the  rank of the maximal order $\cm_A=\prod_{j=1}^t\Z[\zeta_{n_j}]$ which is known by Dirichlet's Unit Theorem.  

In order that  $\cm=\prod_{j=1}^t\Z[\zeta_{n_j}]$ contains an order $\mathcal O$ such that $(\mathcal O^*)_{tors}\cong T$, 
the $n_j$'s must fulfill  the following  necessary conditions 
(see  Lemma \ref{epsilon} below): 
\begin{itemize}
\item[\rm i)] $t\ge \rho+\sigma$;
\item[\rm ii)] $2^\epsilon\mid n_j$ for all $j=1,\dots,t$;
\item[\rm iii)] for each $i=1,\dots, s$ there exists an index $j_i\in\{1,\dots, t\}$ such that   $p_i^{a_i}| n_{j_i}$; moreover, $j_i\ne j_h$ if $p_i=p_h$ and $i\ne h$;
\item[\rm iv)] for each $\iota=1,\dots, \rho$ there exists an index $l_\iota\in\{1,\dots, t\}$ such that   $2^{\epsilon_\iota}| n_{l_\iota}$ and $l_\iota\ne l_h$ if $\iota\ne h.$
\end{itemize}
We will say that the maximal order $\cm=\prod_{j=1}^t\Z[\zeta_{n_j}]$ is  {\it $T$-admissible} if $\{n_1,\dots, n_t\}$ fulfills the conditions (i)-(iv), where the parameters are those of the ``standard'' decomposition of $T$.  

Define
\begin{equation}
\label{m0T}
\cm_{0,T}=\prod_{i=1}^s\Z[\zeta_{2^{\epsilon}p_i^{a_i}}]\times\prod_{\iota=1}^\rho\Z[\zeta_{2^{\epsilon_\iota}}]\times\Z[\zeta_{2^{\epsilon}}]^d,
\end{equation}
where $d=\max\{\sigma-s, 0\}$.  
$\cm_{0,T}$ is $T$-admissible and in Proposition \ref{gT} we prove that $\cm_{0,T}^*$ has  minimum rank among the groups of units of all $T$-admissible maximal orders.
This ensures that 
$$\rank(A^*)=\rank(\cm_A^*)\ge\rank(\cm_{0,T}^*).$$
Now, 
$$\rank(\cm_{0,T}^*)= \sum_{i=1}^s(\frac{\phi(2^{\epsilon} p_i^{a_i})}2-1)+ \sum_{\iota=1}^\rho(\frac{\phi(2^{\epsilon_\iota})}2-1) +d (\frac{\phi(2^{\epsilon} )}2-1)^*,$$
hence
$$\rank(\cm_{0,T}^*)=
\begin{cases}
g(T)&{\rm for}\ \sigma\ge s_0\\
g(T)-(\frac{\phi(2^{\epsilon} )}2-1)^*&{\rm for}\ \sigma< s_0.
\end{cases}
$$
If $ \sigma\ge s_0$ or if $\epsilon=1$ we get the required bound on $\rank(A^*)$.

On the other hand, by Proposition \ref{nom0T} if $\sigma< s_0$, then $\cm_{0,T}$ does not contain any order $A$ with $(A^*)_{tors}\cong T$, so $\cm_A\ne\cm_{0,T}$. 
Now, by Proposition \ref{gT}, for $\epsilon>1$, $\cm_{0,T}$ is the only $T$-admissible maximal order of minimum rank , 
 hence, if   $\sigma< s_0$ and $\epsilon>1$, then $\rank(A^*)>\rank(\cm_{0,T}^*)$ and, using again Proposition \ref{gT}, we get  
 $$\rank(A^*)\ge\rank(\cm_{0,T}^*)+(\frac{\phi(2^{\epsilon} )}2-1)^*=g(T).$$ 
 \qed
\medskip

We now state and prove  the results quoted above.
\begin{lemma}
\label{epsilon}
{\sl ,
Let $\cm=\prod_{j=1}^t\Z[\zeta_{n_j}]$. 
%and assume that  
If $\cm$ contains a subring $A$ with $(A^*)_{tors}\cong T$, then $\cm$ is $T$-admissible. 
%Using for the decomposition of $T$ the ``standard" notation as  in \eqref{eqT}, the following hold:
%\begin{itemize}
%\item[\rm i)] $t\ge \rho+\sigma$;
%\item[\rm ii)] $2^\epsilon\mid n_j$ for all $j=1,\dots,t$;
%\item[\rm iii)] for each $i=1,\dots, s$ there exists an index $j_i\in\{1,\dots, t\}$ such that   $p_i^{a_i}| n_{j_i}$; moreover, $j_i\ne j_h$ if $p_i=p_h$ and $i\ne h$;
%\item[\rm iv)] for each $\iota=1,\dots, \rho$ there exists an index $l_\iota\in\{1,\dots, t\}$ such that   $2^{\epsilon_\iota}| n_{l_\iota}$ and $l_\iota\ne l_h$ if $\iota\ne h.$
%\end{itemize}
}
\end{lemma}
\begin{proof} For each prime $q$, the $q$-Sylow subgroup of $\cm^*$ is the direct product of the (cyclic) $q$-Sylow subgroups of its cyclic factors $\langle\zeta_{n_j}\rangle$, hence every of its $q$-Sylow has at most $t$ cyclic components. Looking at the $2$-Sylow of $T$ we get $t\ge\sigma+\rho$, proving (i). Moreover, if $T$ has an element of order $q^k$, for some $k\ge1$, then the $q$-Sylow of $\cm^*$ has a cyclic component of order at least $q^k$, namely, $q^k|n_j$ for some $j\in\{1,\dots,t\}$; this proves the first part of (iii) and (iv). The last part of these statements follows  by noticing that the $q$-Sylow of $\langle\zeta_{n_j}\rangle$ is cyclic.

We are now left to prove (ii). 
By identifying $A$ with its image in $\prod_{j=1}^t\Z[\zeta_{n_j}]$, we have that the opposite $(-1,\dots,-1)$  of the identity is an element of order 2 in $(A^*)_{tors}= T$ which is in turn a subgroup of $\prod_{j=1}^t\langle\zeta_{n_j}\rangle$. Now, the $2$-Sylow of $(A^*)_{tors}$
is isomorphic to $C_{2^{\epsilon}}^\sigma\times\prod_{\iota=1}^\rho C_{2^{\epsilon_\iota}}$  and all the elements of order 2 of such a group belong to the subgroup $(C_{2^{\epsilon}}^{2^{\epsilon-1}})^\sigma\times\prod_{\iota=1}^\rho C_{2^{\epsilon_\iota}}^{2^{\epsilon_\iota-1}}$, hence they are $2^{\epsilon-1}$-powers since $\epsilon_\iota>\epsilon$ for all $\iota$.
In particular,  
$$(-1,\dots,-1) = \gamma^{2^{\epsilon-1}}=(\gamma_1^{2^{\epsilon-1}},\dots,\gamma_t^{2^{\epsilon-1}})$$ 
with $\gamma_j\in\langle\zeta_{n_j}\rangle$, $\forall\ j$. It follows that $\ord(\gamma_j)={2^\epsilon}$ since  $\ord(\gamma_j)\mid{2^\epsilon}$  and $\ord(\gamma_j)\nmid2^{\epsilon-1}$, 
 so $2^\epsilon | n_j$ for all $j$.
 \end{proof}
 \begin{remark}
 \label{rem-algebra}
{\rm
According to point (ii) of the definition of $T$-admissible maximal order, each $T$-admissible maximal order is a $\Z[\zeta_{2^\epsilon}]$-algebra.
}
\end{remark} 
 \begin{prop}
\label{gT}
{\sl 
Let $\cm=\prod_{j=1}^t\Z[\zeta_{n_j}]$ be $T$-admissible. Then, 
$$\rank(\cm^*)\ge  \sum_{i=1}^s(\frac{\phi(2^{\epsilon} p_i^{a_i})}2-1)+ \sum_{\iota=1}^\rho(\frac{\phi(2^{\epsilon_\iota})}2-1) +d (\frac{\phi(2^{\epsilon} )}2-1)^*$$
 and equality holds only for $\cm=\cm_{0,T}$ or, in the case when $\epsilon=1$, for $\cm=\cm_{0,T}\times\Z^k$ and $k\ge0$. 
 
 Moreover, if $\cm\ne\cm_{0,T}$, then $\rank(\cm^*)\ge \rank(\cm_{0,T}^*)+(\frac{\phi(2^{\epsilon} )}2-1)^*.$
}
\end{prop}

\begin{proof}
For $\cm=\prod_{j=1}^t\Z[\zeta_{n_j}]$, we have 
% is $T$-admissible. Then
\begin{equation}
\label{rangoM}
\rank(\cm^*)=\sum_{j=1}^t\rank(\Z[\zeta_{n_j}]^*)
=\sum_{j=1}^{t}(\frac{\phi(n_j)}2-1)^*.
\end{equation}
% 
%Our aim is to find  a uniform lower bound for the rank of $\cm^*$  which holds for every  $T$-admissible maximal order. We will  achieve this in two steps
Our first step is to bound the rank of $\cm^*$,  by  estimating from below 
% in a way that is independent on the particular $n_j$ which are involved in the structure of $\cm$ but is uniform for all and depend only on the 
 the summands  $\frac{\phi(n_j)}2-1$ for all $j$, using Lemma \ref{disuguaglianze}.

Since $\cm$ is $T$-admissible,  all the $n_j$'s are divisible at least by $2^\epsilon$ and, up to reordering, we can assume that  $n_1,\dots, n_{\rho}$ are divisible by $2^{\epsilon_1},\dots,2^{\epsilon_\rho}$, respectively.

Now,  $\epsilon_j>\epsilon$ for $j=1,\dots,\rho$, so using the inequality \eqref{eq-rank2} we get 
 %can bound from below the term $\frac{\phi(n_j)}2-1$  by using \eqref{eq-rank2}, obtaining
\begin{equation}
\label{eqrankrho} 
%\sum_{j=1}^\rho
\rank(\Z[\zeta_{n_j}]^*)=\frac{\phi(n_j)}2-1\ge
%\sum_{j=1}^\rho
\sum_{\substack{q \ {\rm odd\ prime}\\
q^e || n_j }}\!(\frac{\phi(2^{\epsilon} q^{e})}2-1)+\frac{\phi(2^{\epsilon_j})}2-1.
\end{equation}
For $j=\rho+1,\dots t$ we can use  \eqref{eq-rank3}, which gives
\begin{equation}
\label{eqrankt}
%\sum_{j=\rho+1}^t
\rank(\Z[\zeta_{n_j}]^*)\ge
%\sum_{j=\rho+1}^t
\sum_{\substack{q \,{\rm odd\ prime}\\
q^e || n_j }}(\frac{\phi(2^{\epsilon} q^{e})}2-1)
\footnote[3]{This inequality holds also if $n_j=2$ since $\rank(\Z[\zeta_{n_j}]^*)=0$ and on the RHS we have an  empty sum which is 0.} 
.\end{equation}
%(in this case necessarily $\epsilon=1$)
These inequalities allow to prove that   
\begin{equation}
\label{eqrankM}
\rank(\cm^*)\ge\sum_{i=1}^s(\frac{\phi(2^{\epsilon} p_i^{a_i})}2-1)+ \sum_{\iota=1}^\rho(\frac{\phi(2^{\epsilon_\iota})}2-1)+  d(\frac{\phi(2^{\epsilon} )}2-1)^*.
\end{equation}
In fact, it is enough to show that each term on the RHS of \eqref{eqrankM} appears at least once in \eqref{eqrankrho} or \eqref{eqrankt}, for some $j$. This is trivially the case for the terms in the second sum since each of them appears in \eqref{eqrankrho}. 

As for the first sum, we note that
%Using these inequalities in \eqref{rangoM}
%%Recalling that $
%%\rank(\cm^*)=\sum_{j=1}^t\rank(\Z[\zeta_{n_j}]^*),$ by summing up \eqref{eqrankrho} and \eqref{eqrankt}, 
%we get a lower bound for $\rank(\cm^*)$ depending on the $n_j$'s. 
%
%Actually, we are   
% looking for a uniform lower bound for the rank of the group of units of {\em every}  $T$-admissible maximal order. For this reason we want to consider only the contribution of  some of the terms appearing in \eqref{eqrankrho}  and \eqref{eqrankt}. 
since $\cm$ is $T$-admissible, then each $p_i^{a_i}$ divides some $n_j$. This ensures that, for all $i$,  the RHS of \eqref{eqrankrho} or \eqref{eqrankt} contains a term of type  $\frac{\phi(2^{\epsilon} p_i^{b_i})}2-1$ with $b_i\ge a_i$: we can estimate this term by $\frac{\phi(2^{\epsilon} p_i^{a_i})}2-1$. 
%\begin{equation}
%\label{eqrankM}
%\rank(\cm^*)\ge\sum_{i=1}^s(\frac{\phi(2^{\epsilon} p_i^{a_i})}2-1)+ \sum_{\iota=1}^\rho(\frac{\phi(2^{\epsilon_\iota})}2-1)+  d(\frac{\phi(2^{\epsilon} )}2-1)^*,
%\end{equation}

Finally, the term  $d(\frac{\phi(2^{\epsilon} )}2-1)^*$ can be explained as follows. 
The two sums on the RHS of \eqref{eqrankM}  involve only $s+\rho$ summands, so they can be obtained by considering  the contribution to the rank of  $\tau\le s+\rho$ of the  $(\Z[\zeta_{n_j}])^*$'s. 
%Since $t\ge s+\rho\ge \tau$, 
We estimate the rank of the $t-\tau$ remaining $(\Z[\zeta_{n_j}])^*$'s simply by 
$$\rank(\Z[\zeta_{n_j}]^*)\ge(\frac{\phi(2^{\epsilon} )}2-1)^*.$$
Since $t-\tau\ge0$  and $t-\tau\ge t-\rho-s\ge \sigma-s$ we have $t-\tau\ge d$ and we get \eqref{eqrankM}.

%If $t-\tau\le0$, then $\sigma-s\le t-\rho-s\le t-\tau<0$ so $d=0$ and we are done.
%Otherwise,  we estimate the rank of the  remaining $(\Z[\zeta_{n_j}])^*$'s simply by 
%$$\rank(\Z[\zeta_{n_j}]^*)\ge(\frac{\phi(2^{\epsilon} )}2-1)^*.$$
%we have Observing that $t-\tau\ge t-\rho-s\ge\sigma-s$ and recalling that $d=\max\{\sigma -s, 0\}$ we get \eqref{eqrankM}.
%%The contribution of the 
%%Now, $\tau\le s+\rho\le t$, so 
%then at least $t-\tau$ of the $n_j$'s do not contribute to the two sums on the RHS of \eqref{eqrankM}, since these sums  involve only $s+\rho$ summands. 
%%right hand side of \eqref{eqrankM} neither with a term  $\frac{\phi(2^{\epsilon_\iota})}2-1$ nor with a term  $\frac{\phi(2^{\epsilon} p_i^{a_i})}2-1$: 
%For these $n_j$ 's (they are at least $t-\rho-s\ge\sigma-s$ of them and $d=\max\{\sigma -s, 0\}$) we use the inequality 
%$$\rank(\Z[\zeta_{n_j}]^*)\ge(\frac{\phi(2^{\epsilon} )}2-1)^*.
%$$

%If $s+\rho<t$ then  at least $t-\rho-s$ of the $n_j$'s do not contribute to the two sums on the RHS of \eqref{eqrankM}, since these sums  involve only $s+\rho$ summands. 
%%right hand side of \eqref{eqrankM} neither with a term  $\frac{\phi(2^{\epsilon_\iota})}2-1$ nor with a term  $\frac{\phi(2^{\epsilon} p_i^{a_i})}2-1$: 
%For these $n_j$ 's (they are at least $t-\rho-s\ge\sigma-s$ of them and $d=\max\{\sigma -s, 0\}$) we use the inequality 
%$$\rank(\Z[\zeta_{n_j}]^*)\ge(\frac{\phi(2^{\epsilon} )}2-1)^*.
%$$
The RHS of \eqref{eqrankM} is equal to the rank of $\cm_{0,T}^*$, so  $\cm_{0,T}^*$ has the minimum possible rank among the groups of units of the $T$-admissible maximal orders.  When $\epsilon=1$,  the same is clearly true for the units of $\cm=\cm_{0,T}\times\Z^k$.

Finally, if $\cm=\prod_{j=1}^t\Z[\zeta_{n_j}]$ is $T$-admissible,  but $\cm\ne\cm_{0,T}$, then either $\cm$ has more direct summands  than
%properly contains 
$\cm_{0,T}$ (hence $t>s+\rho+d$) or  at least one of the following holds:

- $2^\epsilon p_{i_1}^{a_{i_1}}p_{i_2}^{a_{i_2}}|n_j$ for some $j$ and two coprime factors $p_{i_1}^{a_{i_1}},p_{i_2}^{a_{i_2}}$;

- $2^{\epsilon_\iota }p_{i}^{a_{i}}|n_j$ for some $\iota, i$ and $j$,

and in both cases we get $\tau<s+\rho$.

In conclusion we always have
%going trough the argument above, we get that 
$t-\tau>d$, so  on the RHS of \eqref{eqrankM} we have at least one extra   summand of type
$(\frac{\phi(2^{\epsilon} )}2-1)^*$, giving
$$\rank(\cm^*)\ge\rank(\cm_{0,T}^*)+(\frac{\phi(2^{\epsilon} )}2-1)^*.$$
\end{proof}

The last proposition shows that  the group of units of $\cm_{0,T}$ has minimum rank among the $T$-admissible maximal orders.
However, for some $T$,   no  order of $\cm_{0,T}$ has  $T$ as the group of torsion units.

\begin{prop}
\label{nom0T}
{\sl Let $T$ be a finite abelian group of even order with its ``standard" notation.
If $\sigma< s_0$, then $\cm_{0,T}$ contains no order $A$ with $(A^*)_{tors}\cong T$. 
}
\end{prop} 
\begin{proof} In this proof, for brevity, we will write $\cm$ for $\cm_{0,T}$. From \eqref{m0T}, we obtain $(\cm^*)_{tors}\cong T\times C_{2^\epsilon}^{s-\sigma}.$ Assume, by contradiction,  that $\cm$ contains an order $A$ with $(A^*)_{tors}\cong T$.
In the notation of \eqref{eqT} and   \eqref{eqTp},   we have $T\cong\prod_{i=1}^{s_0}T_{p_i}\times T_2$, where
\begin{equation}
\label{eqT1}
T_{p_i}=\prod_{j=1}^{v_i} C_{p_i^{b_{ij}}}\ {\rm and}\ T_2=\prod_{\iota=1}^\rho C_{2^{\epsilon_\iota}}\times C_{2^{\epsilon}}^\sigma
\end{equation}
for some $b_{ij}$'s.

For $i=1,\dots,s_0$, put
 $\cm_{p_i}=\prod_{j=1}^{v_i}\Z[\zeta_{2^{\epsilon}p_i^{b_{ij}}}]$
 and let $\cm_2=\prod_{\iota=1}^{\rho}\Z[\zeta_{2^{\epsilon_\iota}}].$
The condition $\sigma< s_0$ yields $d=0$, hence
 $$\cm\cong\left(\prod_{i=1}^{s_0}\cm_{p_i}\right)\times \cm_2 .$$

We first consider the case when $\rho=0$, so $\cm_2$ is trivial.
 
 For each $i=1,\dots,s_0$ let $\alpha_{p_i}=(\zeta_{p_i},\dots\zeta_{p_i})\in \cm_{p_i}$ and  
 put ${\bm \alpha}=(\alpha_{p_1},\dots,\alpha_{p_{s_0}})\in\cm$. 
 Clearly, $\bm \alpha$ is a unit of $\cm$ of order $p_1\cdots p_{s_0}$, therefore  $\bm \alpha$  also belongs to $A^*$, having $(A^*)_{tors}$  the same $p_i$-Sylow subgroups of $(\cm^*)_{tors}$ for all $i=1,\dots s_0$.
 
Now, if $\varphi_\alpha\colon \Z[x]\to A$ is the substitution homomorphism $x\mapsto\alpha$ we have $\Z[\bm\alpha]\cong \Z[x]/(\ker\varphi_{\bm\alpha})$ and
%, using Lemma \ref{lemma1},  
it is easy to check that  $\ker\varphi_{\bm\alpha}$, which by Lemma \ref{lemma1} is principal and generated by a product of cyclotomic polynomials, is generated by
$$\Phi_{p_1}(x)\dots\Phi_{p_{s_0}}(x),$$
and, the primes $p_i$'s  being distinct, Proposition \ref{prop:zetaalpha} ensures that 
$$
  \Z[\bm\alpha]
  %\cong \Z[x]/(\Phi_{p_1}(x)\dots\Phi_{p_{s_0}}(x))]
  \cong \prod_{i=1}^{s_0}\Z[\zeta_{p_i}].
  $$
% In fact, denoting by $\varphi_\alpha\colon \Z[x]\to A$ the substitution homomorphism $x\mapsto\alpha$ we have $\Z[\bm\alpha]\cong \Z[x]/(\ker\varphi_{\bm\alpha})$ and, using Lemmas \ref{lemma1}  and \ref{lemmadx}, it is easy to check that  $\ker\varphi_{\bm\alpha}=(\Phi_{p_1}(x)\dots\Phi_{p_{s_0}}(x))$. 
% Since last isomorphism is trivial, we are left to consider  the  isomorphism in the middle. Nothing has to be proven for $s_0=1$; for $s_0\ge2$ we proceed by induction. 
%The base step $s_0=2$ follows from the Chinese Remainder Theorem since $(\Phi_{p_1}(x),\Phi_{p_2}(x))=\Z[x]$ (see \new{Corollary} \ref{cycloPsi}).
% For $s_0>2$ we have that $(\Phi_{p_i}(x),\Phi_{p_{s_0}}(x))=\Z[x]$ for all $i=1,\dots, s_0-1$, so there exist $a_i(x),b_i(x)\in \Z[x]$ such that 
% $$a_i(x)\Phi_{p_i}(x)+b_i(x)\Phi_{p_{s_0}}(x)=1.$$
% Multiplying the $s_0-1$ equations we  get 
%  $$a(x)\Phi_{p_1}(x)\cdots\Phi_{p_{s_0-1}}(x)+b(x)\Phi_{p_{s_0}}(x)=1,$$
%  for some $a(x),b(x)\in\Z[x]$, hence the Chinese Remainder Theorem gives 
%  $$ \Z[x]/(\Phi_{p_1}(x)\dots\Phi_{p_{s_0}}(x))\cong\Z[x]/(\Phi_{p_1}(x)\dots\Phi_{p_{s_0-1}}(x))]\times \Z[x]/(\Phi_{p_{s_0}}(x))$$
%  and we can conclude by induction.
and
 $$(\Z[\bm\alpha]^*)_{tors}\cong C_2^{s_0}\times C_{p_1}\times\dots\times C_{p_{s_0}}.$$
This gives a contradiction since $(\Z[\bm\alpha]^*)_{tors} < (A^*)_{tors}\cong  T$ and  $\sigma<s_0$.
 
 In the case when $\rho>0$, we have to slightly modify the previous argument to find a contradiction.
 
 As in the previous case, for each $i=1,\dots,s_0$ let $\alpha_{p_i}=(\zeta_{p_i},\dots\zeta_{p_i})\in \cm_{p_i}$;   also denote by $ v_0$ the unit element of $\cm_2$ and by  $\alpha_2=-v_0=(-1,\dots,-1)$ its opposite.

In $\cm$ consider the elements 
$${\bm \alpha}'=(\alpha_{p_1},\dots,\alpha_{p_{s_0}},v_0), \text{ and } \bm\delta=(1,\dots,1, \alpha_2).$$ 
Both of them belong to $A$: in fact, $\bm \alpha'\in A^*$ since it is a unit of $\cm$ of odd order; $\bm\delta$ is a $2^\epsilon$ power in $\cm^*$ and $(\cm^*)_{tors}^{2^\epsilon}=T^{2^\epsilon}=(A^*)_{tors}^{2^\epsilon}$. 

It follows that also ${\bm \alpha}=(\alpha_{p_1},\dots,\alpha_{p_{s_0}},\alpha_2)=
{\bm \alpha}'\bm\delta$ belongs to $A$, so that $\Z[\bm\alpha]\subseteq A$.
 As before, we have 
 $$\Z[\bm\alpha]\cong \Z[x]/(\Phi_{p_1}(x)\dots\Phi_{p_{s_0}}(x)\Phi_2(x))$$
 and, again by Proposition \ref{prop:zetaalpha}, we  get
 $$\Z[{\bm\alpha}]\cong\Z\times\prod_{i=1}^{s_0}\Z[\zeta_{p_i}]$$
 %(in fact, the previous argument does not require the $p_i$'s to be odd, but only distinct primes), 
and  therefore
  $$(\Z[{\bm\alpha}]^*)_{tors}\cong C_2^{s_0+1}\times C_{p_1}\times\dots\times C_{p_{s_0}}$$
 is a subgroup of $(A^*)_{tors}$.
 
 If $\rho=1$ this gives a contradiction, since  $A^*$ has exactly $\sigma+1$ cyclic factors of order a power of 2, and $s_0>\sigma$.
 
If $\rho>1$, for each $\iota=1,\dots,\rho-1$,  let ${ \bm \beta}_\iota$ be the element of  $\cm$ with all coordinates 1, but whose coordinate in $\Z[\zeta_{2^{\epsilon_\iota}}]$  is  equal to $-1$. These elements generate a subgroup of $(\cm)^*$ isomorphic to $C_{2}^{\rho-1}$. 
Moreover, all the ${\bm \beta}_\iota$'s  belong to $A^*$: in fact, ${\bm\beta}_\iota$ is a $2^\epsilon$ power of an element of $(\cm^*)_{tors}$ so it belongs to $(\cm^*)_{tors}^{2^\epsilon}=(A^*)_{tors}^{2^\epsilon}$. 

Now,  
$$(\Z[{\bm\alpha}]^*)_{tors}\cap\langle{\bm \beta}_1,\dots,{\bm \beta}_{\rho-1}\rangle=\{(1,\dots,1)\},$$ 
in fact, the torsion units of $\Z[{\bm\alpha}]$ are of type $ ((-\alpha_{p_1})^{e_1},\dots,(-\alpha_{p_{s_0}})^{e_{s_0}},\alpha_2^{e_0})$ with $e_0,e_1\dots,e_{s_0}\in\Z$, so their coordinates in $\cm_2$ are all 1 or all -1. It follows that $(A^*)_{tors}$ contains a subgroup isomorphic to 
  $$(\Z[{\bm\alpha}]^*)_{tors}\times\langle{\bm \beta}_1,\dots,{\bm \beta}_{\rho-1}\rangle\cong C_2^{s_0+\rho}\times C_{p_1}\times\dots\times C_{p_{s_0}}$$
  and this is not possible since $A^*$ has $\sigma+\rho$ cyclic factors of order a power of 2, and $s_0+\rho>\sigma+\rho$.
\end{proof}

\subsection{Proof of Theorem \ref{torfree}: the  ``if" part}

Let $T$ be any finite abelian group of even order; consider on $T$ its ``standard'' notation as in  \eqref{eqT}. For each $g\ge g(T)$ we will construct an example of a torsion free ring $A$ with $A^*\cong T\times\Z^g$.

The first and most substantial step is the construction for $g=g(T).$
 The following propositions deal with two particular cases.

\begin{prop}
\label{pgroup}
{\sl
Let $p$ be an odd prime and let $\epsilon, b_1,\dots, b_v$ be  integers, with $1\le b_1\le b_2\le\dots\le b_v$ and $\epsilon\ge1$.
The maximal order $\cm= \prod_{j=1}^{v}\Z[\zeta_{2^{\epsilon}p^{b_j}}]$
 contains {an order} $A$ with $(A^*)_{tors}\cong 
C_{2^\epsilon}\times\prod_{j=1}^{v} C_{p^{b_j}}$.
}
\end{prop}
\begin{proof}
For $j=2,\dots,v,$ let $\bm\beta^{(j)}=(\beta_1^{(j)},\dots, \beta_v^{(j)})\in \cm$, where $\beta_i^{(j)}=1$  for $i\ne j$ and  $\beta_j^{(j)}=\zeta_{p^{b_j}},$ and put 
$$A=\Z[\zeta_{2^\epsilon p^{b_1}} ][\bm\beta^{(2)}, \dots ,\bm\beta^{(v)}]$$
where we are identifying $A$ with a subring of $\cm$ via the diagonal embedding of $\Z[\zeta_{2^\epsilon p^{b_1}}]$. This means that we identify  $\zeta_{2^\epsilon p^{b_1}}$ with ${\bm\alpha}=(\zeta_{2^\epsilon p^{b_1}}, \dots, \zeta_{2^\epsilon p^{b_1}})$.

We claim that $(A^*)_{tors}\cong V=C_{2^\epsilon}\times\prod_{j=1}^{v} C_{p^{b_j}}$.

It is clear that the elements ${\bm\alpha},{\bm\beta}^{(2)},\dots, {\bm\beta}^{(v)}\in A$ are multiplicatively independent units and that they generate a subgroup of $(A^*)_{tors}$ isomorphic to $V$. On the other hand, $(\cm^*)_{tors}\cong \prod_{j=1}^v C_{2^\epsilon p^{b_j}}$ , hence, up to isomorphism,
$$(A^*)_{tors}\le V\times C_{2^\epsilon}^{v-1}.$$ 
To prove that $(A^*)_{tors}\cong V$ it is enough to show that the 2-Sylow of $(A^*)_{tors}$ is cyclic, or equivalently, that $(-1,\dots,-1)$ is the only element of order 2 of $(A^*)_{tors}$.

For each $i=2,\dots, s$ define $\cm_i=\Z[\zeta_{2^\epsilon p^{b_1}}]\times\Z[\zeta_{2^\epsilon p^{b_i}}]$ and  denote by $\pi_i\colon\cm\to\cm_i$ the canonical projection. Put  $A_i=\pi_i(A)$  and  $\beta_{0,i}=(1,\zeta_{p^{b_i}})$, then 
$$A_i=\Z[\zeta_{2^\epsilon p^{b_1}}][\pi_i({\bm\beta}^{(2)}),\dots, \pi_i({\bm\beta}^{(v)})] = \Z[\zeta_{2^\epsilon p^{b_1}}][\beta_{0,i}].$$

Let $\varphi_{\beta_{0,i}}$ be the evaluation homomorphism defined on  $\Z[\zeta_{2^\epsilon}][x]$. It is easily checked that
its kernel is 
%of the evaluation homomorphism $\varphi_{\beta_{0,i}}$ defined on  $\Z[\zeta_{2^\epsilon}][x]$ is \ne{easily checked to be } 
generated by $(x-1)\Psi_{p^{b_i}, p^{b_1}}(x)$, so 
$$
A_i=\Z[\zeta_{2^\epsilon p^{b_1}}][\beta_{0,i}]\cong\Z[\zeta_{2^\epsilon p^{b_1}}][x]/((x-1)\Psi_{p^{b_i}, p^{b_1}}(x))
$$
     and, by Proposition \ref{units_1n}, $(A_i^*)_{tors}\cong C_{2^\epsilon p^{b_1}}\times C_{p^{b_i}}$.
     This ensures that, for all indices $i$,  the 2-Sylow of  $\pi_i((A^*)_{tors})$, which is a subgroup of  $(A_i^*)_{tors}$, is cyclic and this allows us to conclude the proof. 
In fact,
let  ${\bf u}=(u_1,\dots,u_v)\in \cm^*$ be such that ${\bf u}^2=(1,\dots,1)$; if  ${\bf u}\in A$, then $\pi_i({\bf u})=(u_1,u_v)$ is an element of exponent 2 of $(A_i^*)_{tors}$, so    $(u_1,u_v)$ must be equal to $(1,1)$ or $(-1,-1)$, in particular, $u_i=u_1$ for all $i=1,\dots,v$. This yields    ${\bf u}=(1,\dots,1)$ or  ${\bf u}=(-1,\dots,-1)$, so $A^*$ has only one element of order 2,  therefore
$$(A^*)_{tors}\cong C_{2^\epsilon}\times C_{p^{b_1}}\times\dots \times C_{p^{b_v}}=V.$$
\end{proof}

When the group $T$ has too few 2-cyclic factors of minimal order,  Proposition \ref{nom0T} shows that no order of $\cm_{0,T}$, has torsion units isomorphic to $T$. In this case, to find an order $A$ with  $(A^*)_{tors}\cong T$,  we have to consider a bigger maximal order obtained by adding to  $\cm_{0,T}$ an extra direct factor, which works as a ``control'' factor on the 2-torsion. The following proposition deals with the case $\sigma=1$.
\begin{prop}
\label{sigma=1}
{\sl
Let $p_1,\dots p_s$ be prime numbers and let $\epsilon, a_1,\dots a_s$ be positive integers.
The maximal order 
$\cm= \Z[\zeta_{2^\epsilon}] \times \prod_{i=1}^s \Z[\zeta_{2^\epsilon p_i^{a_i}}]$ contains a subring $A$ with $(A^*)_{tors}\cong C_{2^\epsilon}\times\prod_{i=1}^s C_{p_i^{a_i}}$.
}
\end{prop}
\begin{proof}
For each $i=1,\dots,s$, let ${\bm{\beta}}^{(i)}=(1,\beta_1^{(i)},\dots,\beta_s^{(i)})\in\cm$, where  $\beta_j^{(i)}=1$ for all $j\ne i$ and  $\beta_i^{(i)}=\zeta_{p_i^{a_i}}$ . Put  $$A=\Z[\zeta_{2^\epsilon}][{\bm\beta}^{(1)},\dots, {\bm\beta}^{(s)}]$$
viewed  as a subring of $\cm$.
We claim that $(A^*)_{tors}\cong C_{2^\epsilon}\times\prod_{i=1}^s C_{p_i^{a_i}}$.

Clearly, the elements ${\bm\alpha}=(\zeta_{2^\epsilon},\dots,\zeta_{2^\epsilon}),{\bm\beta}^{(1)},\dots, {\bm\beta}^{(s)}\in A$ are multiplicatively independent units which generate a subgroup of $(A^*)_{tors}$ isomorphic to $C_{2^\epsilon}\times\prod_{i=1}^s C_{p_i^{a_i}}$.

On the other hand, $(\cm^*)_{tors}\cong C_{2^\epsilon}^{s+1}\times \prod_{i=1}^s C_{p_i^{a_i}}$ , then
to prove our claim it is enough to show that the 2-Sylow of $(A^*)_{tors}$ is cyclic, or equivalently that $(-1,\dots,-1)$ is the only element of order 2 of $(A^*)_{tors}$. 

This can be proved arguing  as in the previous proposition. In fact,
for each $i=1,\dots, s$ define $\cm_i=\Z[\zeta_{2^\epsilon}]\times\Z[\zeta_{2^\epsilon p_i^{a_i}}]$ and  denote by $\pi_i\colon\cm\to\cm_i$ the canonical projection. 
Let  $A_i=\pi_i(A)$  and  $\beta_{0,i}=(1,\zeta_{p_i^{a_i}})$, then 
$$A_i=\Z[\zeta_{2^\epsilon}][\pi_i({\bm\beta}^{(1)}),\dots, \pi_i({\bm\beta}^{(s)})] = \Z[\zeta_{2^\epsilon}][\beta_{0,i}].$$
The kernel of the evaluation homomorphism $\varphi_{\beta_{0,i}}\colon\Z[\zeta_{2^\epsilon}][x]\to A$ is generated by $\Phi_1(x)\Phi_{p_i^{a_i}}(x)$: in fact, since $p_i^{a_i}$ is odd, the polynomial  $\Phi_{p_i^{a_i}}(x)$ is irreducible in $\Z[\zeta_{2^\epsilon}]$. 
Thus    $$A_i=\Z[\zeta_{2^\epsilon}][\beta_{0,i}]\cong\Z[\zeta_{2^\epsilon}][x]/(\Phi_1(x)\Phi_{p_i^{a_i}}(x))$$
     and, by Proposition \ref{units_1n}, $(A_i^*)_{tors}\cong C_{2^\epsilon p_i^{a_i}}$.
This implies that  also its subgroup  $\pi_i((A^*)_{tors})$ is cyclic and this allows us to conclude the proof. 

In fact,
let  ${\bf u}=(u_0,\dots,u_s)\in \cm^*$ be such that ${\bf u}^2=(1,\dots,1)$; if  ${\bf u}\in A$, then, for all $i$,  $\pi_i({\bf u})=(u_0,u_i)$ is an element of exponent 2 of the cyclic group $\pi_i((A^*)_{tors})$, so    $(u_0,u_i)$ must be equal to $(1,1)$ or $(-1,-1)$. In particular, $u_i=u_0$ for all $i=1,\dots,s$. This ensures that    ${\bf u}=(1,\dots,1)$ or  ${\bf u}=(-1,\dots,-1)$, and $A^*$ has only one element of order 2, as required.
\end{proof}

We are now ready for the general construction for $g=g(T).$

Let 

\begin{equation}
\label{mT}
\cm_T=\begin{cases}
	\cm_{0,T}&{\rm for\ }\sigma\ge s_0\\
	\cm_{0,T}\times \Z[\zeta_{2^\epsilon}]&{\rm for\ }\sigma<s_0,
\end{cases}
\end{equation}
then $\rank(\cm_T)=g(T)$ for all $T$. We will construct  $A$ as an order in $\cm_T$. 

The case when $s\le\sigma$ is very easy: we can simply take $A=\cm_T$  since $\cm_T^*\cong T\times\Z^{g(T)}$. 

Consider now the more general case when $\sigma\ge  s_0$. We can write the group $T$ as
$$T= V_2\times\prod_{i=1}^{s_0}V_{p_i},$$
where $V_{p_i}=C_{2^\epsilon}\times T_{p_i}=C_{2^\epsilon}\times\prod_{j=1}^{v_i} C_{p_i^{b_{ij}}}$ and $V_2=C_{2^\epsilon}^{\sigma-s_0}\times\prod_{\iota=1}^\rho C_{2^{\epsilon_\iota}}$.

For $i=1,\dots,s_0$, let 
 $\cm_{p_i}=\prod_{j=1}^{v_i}\Z[\zeta_{2^{\epsilon}p_i^{b_{ij}}}]$
 and  $\cm_2=\Z[\zeta_{2^\epsilon}]^{\sigma-s_0}\times\prod_{i=1}^{\rho}\Z[\zeta_{2^{\epsilon_i}}].$
 Then 
 $$\cm_T\cong \cm_2\times\prod_{i=1}^{s_0}\cm_{p_i}.$$
 
By Proposition \ref{pgroup},  for all $p= p_1,\dots, p_{s_0}$, the maximal order $\cm_p$ contains an order  $A_p$ such that $(A_p^*)_{tors}\cong V_p$. It follows that $A=\cm_2\times\prod_{i=1}^{s_0}A_{p_i}$ is an order of $\cm_T$ with $(A^*)_{tors}\cong T$.

Let now $\sigma< s_0$. 
We write the group $T$ as  $T_0\times T_1$ where
$$T_0=\prod_{i=1}^{\sigma-1}C_{2^\epsilon p_i^{a_i}}\times\prod_{\iota=1}^\rho C_{2^{\epsilon_\iota}}\ {\rm and}\  T_1=C_{2^\epsilon}\times C_{p_{\sigma}^{a_\sigma}}\times\cdots\times C_{p_{s}^{a_s}}.$$

By Proposition \ref{sigma=1} the order 
$\cm_1= \Z[\zeta_{2^\epsilon}] \times \prod_{i=\sigma}^s \Z[\zeta_{2^\epsilon p_i^{a_i}}]$ contains a subring $A_1$ with $(A_1^*)_{tors}\cong T_1$.

On the other hand, 
$$\cm_T=\cm_1\times\prod_{i=1}^{\sigma-1}\Z[\zeta_{2^\epsilon p_i^{a_i}}]\times \prod_{i=1}^\rho\Z[\zeta_{2^\epsilon_i}]$$
and its subring
$$A=A_1\times\prod_{i=1}^{\sigma-1}\Z[\zeta_{2^\epsilon p_i^{a_i}}]\times \prod_{i=1}^\rho\Z[\zeta_{2^\epsilon_i}]$$
is such that $(A^*)_{tors}\cong T.$ 

Moreover, $\rank(A^*)\le\rank(\cm_T^*)$. On the other hand,  the rank of $A^*$ is the same of the
rank of ${\mathcal M}^*_A$, which is a $T$-admissible maximal order, and thus its rank is at least
the rank of ${\mathcal M}_T$. This gives $\rank(A^*)=\rank(\cm_T^*)$ and also proves that $A$ is an order of $\cm_T$.

\smallskip

The final step is the construction of torsion-free rings with  group of units isomorphic to $T\times\Z^{g}$  for all  $g>g(T)$. Also in this case if  $A$ is  a torsion-free ring with  
$(A^*)_{tors}=T$ and minimal rank $g(T)$, then $\mathcal A=A[x_1,\dots, x_k,x_1^{-1},\dots, x_k^{-1}]$ is  torsion-free and  has group of units isomorphic to  $ T\times \Z^{g(T)+k}$.

\section{Reduced rings}
\label{reduced}
In this section we classify the finitely generated abelian groups which arise as groups of units of  reduced rings.
The next  proposition describes the relation between the units of a ring and those of its reduced quotient, showing that the study of reduced rings is a substantial step to the study of units of a general ring. 
%a certain (weak) sense it is permitted to
%restrict attention to reduced rings. 
\begin{prop}
\label{successioneesatta}
{\sl
Let $A$ be a commutative ring and let $\mathfrak{N}$ be its nilradical. Then the sequence
\begin{equation}
\label{success}
1\to 1+\mathfrak{N}\hookrightarrow A^*\stackrel{\phi}{\to}\left(A/\mathfrak{N}\right)^*\to 1,
\end{equation}
where $\phi(x)=x+\mathfrak{N}$, is exact.
}
\end{prop}
%\begin{proof} Clearly  
%$$ 1\to 1+\mathfrak{N}\to A^*\to{A^*}/{(1+\mathfrak{N})}\to 1$$
%is exact. Moreover, the map $\Phi : A^*\to (A/\mathfrak{N})^*$  defined by $x\mapsto x+\mathfrak{N}$ is a homomorphism with kernel equal to $1+\mathfrak{N}$. Further, if $x+\mathfrak{N}\in (A/\mathfrak{N})^*$ and $(x+\mathfrak{N})( y+\mathfrak{N})=1+\mathfrak{N}$ then $xy\in1+\mathfrak{N}\subseteq A^*$, so $x\in A^*$ and $\Phi$ is surjective, so $A^*/(1+\mathfrak{N})\cong( A/\mathfrak{N})^*.$
%\end{proof}

We note that  for finite characteristic rings the exact sequence  \eqref{success} always splits (see \cite[Thm 3.1]{DCDampa}). This is no longer true in general, as shown in \cite[Ex 2]{dcdBLMS}).

\smallskip

The units of a reduced ring of finite characteristic rings are characterized as follows.
\label{positive-char}

\begin{prop}
  \label{reduced-finiti} {\sl 
  The finitely generated abelian groups which are the groups of units of reduced rings of positive characteristic 
  are exactly those of the form
$$\prod_{i=1}^k\F_{p_i^{n_i}}^*\times \Z^g$$
where $k, n_1,\dots, n_k$ are positive integers, $\{p_1,\dots,p_k\}$ are not necessarily distinct prime numbers
and $g\ge0.$
}
\end{prop}
\begin{proof}
Let $A$ be a reduced ring of characteristic $n$, such that  $A^*\cong (A^*)_{tors}\times \Z^g$, with $(A^*)_{tors}$  finite and $g\ge0$. The ring  $B=\Z/n\Z[(A^*)_{tors}]$ is a finite  ring and by Lemma \ref{minimum} $B^*=(A^*)_{tors}$. Since $B$ is finite, $B$ is artinian and so it is a product of local  artinian rings. Moreover,  a reduced local artinian ring is a field, hence $B$ is a product of finite fields  (see also \cite[Corollary 3.2]{DCDampa}) and we get that $(A^*)_{tors}=B^*$ has the required form. 

On the other hand, let the $p_i$'s, $n_i$'s and $g$ be as in the statement and put $R=\prod_{i=1}^k\F_{p_i^{n_i}}.$ Then the ring $R[x_1,\dots,x_g,x_1^{-1},\dots,x_g^{-1}]$ has group of units isomorphic to  $\prod_{i=1}^k\F_{p_i^{n_i}}^*\times \Z^g$. 
\end{proof}
The following proposition together with the results of the previous section allows us to classify the finitely generated abelian groups which arise as group of units of a reduced ring.

\begin{prop}{\rm (\cite[Prop.~1]{PearsonSchneider70})}
\label{ps} 
{\sl Let $A$ be a commutative ring which is finitely generated and integral over its fundamental subring. Then $A=A_1\oplus A_2$, where $A_1$ is a finite ring and the torsion ideal of $A_2$ is contained in its nilradical.}
\end{prop}

Now, if $A$ is reduced then the finite ring $A_1$ is reduced and $A_2$ is torsion-free. Then, Theorems \ref{torfree} and \ref{reduced-finiti} immediately gives the following.

\begin{theo}
\label{reduced-all}
{\sl The  finitely generated abelian groups that occur as groups of units of  reduced rings are those of the form
$$\prod_{i=1}^k\F_{p_i^{n_i}}^*\times T\times \Z^g$$
where $k, n_1,\dots, n_k$ are positive integers, $\{p_1,\dots,p_k\}$ are not necessarily distinct prime numbers, $T$ is any finite abelian group of even order and $g\ge g(T)$.
}
\end{theo}

\bibliographystyle{amsalpha}

\bibliography{biblio}
\end{document}